\documentclass[aap]{imsart}

\RequirePackage{amsthm,amsmath,amsfonts,amssymb}
\RequirePackage[numbers]{natbib}
\RequirePackage[colorlinks,citecolor=blue,urlcolor=blue]{hyperref}
\RequirePackage{graphicx, color} 
\RequirePackage{comment}

\startlocaldefs
\RequirePackage{dsfont}
\RequirePackage{thmtools}
\RequirePackage{tikz}
\RequirePackage{forest}
\usetikzlibrary{arrows.meta, positioning, calc}
\RequirePackage{float}

\RequirePackage{mathtools}
\RequirePackage{hyperref}
\RequirePackage{cleveref}
\RequirePackage{mathrsfs}
\RequirePackage[english]{babel}
\RequirePackage{csquotes}
\RequirePackage{bm}
\RequirePackage{tikz-cd}
\RequirePackage[dvipsnames]{xcolor} 

\RequirePackage{ulem}
\RequirePackage{enumitem}

\usepackage{tikz}
\usepackage{pgfplots}
\pgfplotsset{compat=1.18}

\theoremstyle{plain}
\newtheorem{thm}{Theorem}
\newtheorem{lm}[thm]{Lemma}
\newtheorem{cor}[thm]{Corollary}
\newtheorem{prop}[thm]{Proposition}

\theoremstyle{definition}
\newtheorem{defn}[thm]{Definition}

\newtheorem{remark}[thm]{Remark}

\numberwithin{equation}{section}

\RequirePackage{xcolor}

\newcommand{\N}{\mathbb{N}}
\newcommand{\Z}{\mathbb{Z}}

\newcommand{\R}{\mathbb{R}}

\newcommand{\e}{\mathrm{e}}
\newcommand{\order}{\mathcal{O}}
\newcommand{\Cf}{\mathcal{C}}

\endlocaldefs

\begin{document}

\begin{frontmatter}
\title{A Model for Fast Mining Blockchains}
\runtitle{A Model for Fast Mining Blockchains}

\begin{aug}
\author[A]{\fnms{Rhys}~\snm{Bowden}\ead[label=e1]{Rhys.Bowden@monash.edu}}

\author[B]{\fnms{Cengiz}~\snm{Gazi}\ead[label=e2]{c.gazi1@student.unimelb.edu.au}}

\author[B]{\fnms{Jan}~\snm{de Gier}\ead[label=e3]{jdgier@unimelb.edu.au}}

\author[B]{\fnms{Peter}~\snm{Taylor}\ead[label=e4]{taylorpg@unimelb.edu.au}}

\address[A]{School of Public Health and Preventive Medicine, Monash University, Australia\printead[presep={,\ }]{e1}}

\address[B]{School of Mathematics and Statistics, University of Melbourne, Victoria 3010, Australia\printead[presep={,\ }]{e2,e3,e4}}

\end{aug}

\begin{abstract}
We study a Markovian model for a blockchain mining process. In the limit where there is an infinite number of miners, the model satisfies a functional law of large numbers due to Barbour and Luczac (2008), leading to an infinite set of coupled ODEs characterising the blockchain heights. We show that the sequence of solutions to this system converges uniformly, to a limiting profile in an appropriate comoving frame. Utilising asymptotic arguments inspired by Brunet and Derrida (1997), we derive an explicit expression for the position of the travelling front, which exhibits a leading order logarithmic Bramson correction in time. For a finite number of miners, we further derive a correction to the travelling speed which is logarithmic in the number of miners. Finally, by application of a computation due to Ablowitz and Zeppetella (1979), we identify that, for a specific choice of the travelling speed, the \textit{Painlev\'e} integrable, travelling wave solution of the Fisher-KPP equation with drift accurately matches our limiting profile. Our results are supported by numerical simulations.
\end{abstract}



\end{frontmatter}


\section{Introduction}

A blockchain is constructed by miners each of whom works to discover new blocks according to a process that is well-approximated by a Poisson process \cite{peter_blockchains_2018}, with a rate that is proportional to the miner's computational power. Each block that is mined must reference a parent block and that block in turn must reference its parent block all the way back to the original block, the genesis block. The height of a block is the number of ancestor blocks that lie between it and the genesis block, and the length of a branch is determined by the block of maximum height. We shall refer to the longest branch that an individual miner knows about as its {\it local blockchain}. If a miner knows of two or more different branches of the same length then its local blockchain is given by the one that it knew about first. At any given time a miner is working to append a new block to its local blockchain. We say that it is mining \textit{on top of its local blockchain}.

A miner's local blockchain can progress in two ways.
\begin{enumerate}
\item The miner can discover a new block. In this case, it appends the block to its local blockchain and copies the block to the other miners via a peer-to-peer network of nodes. 
\item The miner can receive a block copied from another miner. In this case, the miner will validate the proof-of-work and the transactions in the block. If these are in order and if the height of the newly-received block is greater than the height of the block at the head of its local blockchain, it will acquire previous blocks if necessary and institute the new branch as its local blockchain, That is, it will start mining {\it on top of the new block}. It will also forward the new block to other adjacent nodes.
\end{enumerate}

If there are two or more different branches of the same length present in the network, then at least two different miners must have different local blockchains. In a slow mining blockchain, this situation is resolved when a miner discovers a new block on top of one of the versions and propagates this new block to the rest of the miners. This has a high probability of occurring before any other miner discovers a new block. Our interest in this paper is to consider what will happen in a {\it fast mining} blockchain where the time that it takes to mine a new block is of the same order as the time that it takes to communicate the block to a second miner and for that miner to validate it - we call this the {\it propagation delay}. In this case, it is likely that other blocks will be mined while newly discovered blocks are being propagated and it is not clear that the local blockchains at the different miners will ever be the same.

We build on previous models of the blockchain forking process \cite{peter_blockchains_2015, peter_blockchains_2018} to consider fast mining blockchains. We develop a simple continuous-time Markov chain (CTMC) model for this case to examine whether the community of miners will ever reach consensus. \Cref{prop: consensus positive reccurent} shows that no matter how far the blockchains at the individual miners diverge, the local blockchains will almost surely return to consensus in finite expected time. Furthermore, even as the mining delay becomes so short that consensus is only rarely reached, we demonstrate using a functional law of large numbers that miners typically agree on all blocks except those that lie within a bounded distance of the head of the longest local blockchain.

Specifically, in the limit of infinitely many miners, our fast mining blockchain model satisfies a functional law of large numbers, \cite[Theorem 3.1]{BarbourLuczak_2008}, and hence the limiting miner density profile satisfies a deterministic system of ODEs. The limiting distribution function of miner heights then satisfies the nonlinear differential-difference equation (\Cref{prop:BC})
\begin{equation}\label{eq:BC0}
    \left\{
    \begin{aligned}
    \dot{\varphi}_n(t) &= -\lambda(\varphi_n - \varphi_{n-1}) - \omega \varphi_n(1-\varphi_n), \quad &&n \geq 0,\\
    \varphi_n(0) &= 1, \quad &&n \geq 0,
    \end{aligned}
    \right.
\end{equation}
with boundary condition $\varphi_{-1} \equiv 0$. This system of equations forms the basis of study for the rest of the paper. We note here that our system only includes a first order difference and that a related system with second order difference was studied in \cite{discreteFKPP_BramsonCorrection_2023}.

Convergence of the limiting distribution function profile in the large height and large time limit is given by \Cref{thm: u_n convergence} in \Cref{sec: Convergence to limiting profile}. This result is proved using a comoving frame and technical uniform estimates. This limiting profile is a traveling wave profile with shape defined implicitly by \Cref{thm: u_n convergence}, and admits the universal leading order logarithmic Bramson correction to the wave speed. We demonstrate this via an asymptotic argument originally due to \cite{BrunetDerrida_1997}.  By comparison with a continuum approximation and application of Painlev\'e integrability for a special choice of rates \cite{AblowitzZeppetella_1979} we provide an explicit approximate analytic formula for the shape of the limiting profile. We confirm a good fit with the corresponding solution of \eqref{eq:BC0} by numerical comparison.
 
The convergence of the travelling wave speed of the finite $N$-miner process to the travelling wave speed of the limiting $N=\infty$-miner process is shown to exhibit Brunet-Derrida behaviour, that is a logarithmic correction in $N$ first observed in \cite{BrunetDerrida_1997} for a cut-off modified Fisher-KPP system.

\subsection{Outline of paper} The rest of the paper is organised as follows: In \Cref{sec: A blockchain mining model} we formally define our blockchain mining model and prove consensus is positive recurrent for any finite number of miners. In \Cref{sec: height dynamics}, we show that that the height dynamics of our blockchain model, under a normalised occupation representation, satisfies a functional law of large numbers. We then derive and state the set of equations \eqref{eq:BC} which form the basis for the analysis of the rest of the paper. In \Cref{sec: Convergence to limiting profile}, we show that the sequence of solutions to \eqref{eq:BC} converge uniformly to a limiting profile, in an appropriate comoving frame - this result is stated formally in \Cref{thm: u_n convergence}. In \Cref{sec: speed and shape of limiting profile}, we derive an explicit expression for the position of the travelling front and its leading order logarithmic Bramson correction using an asymptotic analysis \cite{BrunetDerrida_1997}. We then provide an approximate analytic expression for the limiting profile. We obtain this expression using an asymptotic argument that reduces our set of equations \eqref{eq:BC} to a travelling wave Fisher-KPP equation with drift, and then applying a result from \cite{AblowitzZeppetella_1979}, which gives an explicit analytic solution made possible by theory concerning Painlev\'e integrability. In \Cref{sec: shift in speed for finite miners}, we derive an explicit expression for the shift in speed of the travelling wave for a finite number of miners compared to that obtained in the infinite miner limit. Via an application of a perturbation computation from \cite{BrunetDerrida_1997}, this leads to a logarithmic Brunet-Derrida correction.

\section{Model motivation:  Blockchain mining}\label{sec: A blockchain mining model}

\subsection{Our model} Our model of the blockchain for $N$ miners is expressed in terms of a continuous-time Markov chain (CTMC) with states
\begin{equation*}
\bm{B} = (B_1, B_2, \ldots, B_N) \in \mathbb{B}^N, \quad N \geq 2,
\end{equation*}
with each component taking values in the space of positive increasing finite integer sequences with first value equal to $1$, that is
\begin{equation*}
\mathbb{B} := \left\{ B \in \bigcup_{n \geq 1} \N^n \mid 1 = b^1 < b^2 \ldots < b^n \right\},
\end{equation*}
where $b^j$ is the $j^{th}$ sequential component of $B$.

We interpret $B_i$ as representing the sequence of connected blocks of the $i^{th}$ miner, which we call its {\it local blockchain}, and assume all miners begin mining on top of the first block, called the {\it genesis block}, with block label $b^1=1$. We say this local blockchain has $|B_i|$ blocks with distance from the genesis block equal to $|B_i|-1$, which we call  its {\it height}. The space $\mathbb{B}$ represents the space of all possible local blockchains.

\subsubsection{Mining} Each miner is continuously trying to mine the next block, with successes occurring according to a Poisson process with parameter $\lambda>0$. It follows that, at any given time $t \geq 0$, the total number of blocks $L(t) \in \N$, equivalently the label of the most recently mined block, is distributed as
\begin{equation*}
L(t) \sim 1 + \mathrm{Poisson}(N\lambda t).
\end{equation*}

\subsubsection{Block propagation} Each miner continuously communicates their {\it local blockchain} to the other $N-1$ miners according to a Poisson process with parameter $\omega>0$. Whenever such a communication event occurs, with uniform probability $\tfrac{1}{N-1}$ a given one of the other $N-1$ miners receives the communication. In this case, it will perform the necessary proof-of-work tests, and if the newly received blockchain is both valid, and has height greater than the height of the receiving miner's current local blockchain, it will abandon its local blockchain and start mining on top of this newly received local blockchain. We call this a {\it block propagation event}. Any blocks that were in the receiving miner's abandoned local blockchain that are not present elsewhere in the network are called {\it orphan blocks}. They cannot be present in the long-term blockchain. 

To illustrate, if the $j^{th}$ miner sends the communication and the $i^{th}$ miner receives the communication, then we say a block propagation event occurs from miner $j$ to miner $i$, if the height of the local blockchain of miner $j$ is greater than that of miner $i$, that is, $|B_j|>|B_i|$. In this case, we say that miner $j$ has {\it propagated} its local blockchain to miner $i$. If $|B_j| \leq |B_i|$, then miner $i$ ignores the communication from miner $j$ and will continue to mine on its current local blockchain.


\subsubsection{Consensus} We say the network of miners $\bm{B}$ is {\it in consensus} if 
\begin{align}\label{eq: consensus Bi=Bj}
    B_i = B_j, \quad \text{for all $i,j \in [N]$},
\end{align}
that is, all local blockchains are identical.

\subsubsection{Blocktree} There is always a connected sequence of at least one block at the beginning of each local blockchain that every miner agrees on. It follows that each state $\bm{B}$ corresponds to a rooted tree, which we call the {\it Blocktree}.

\subsection{Example: $N=7$ miners} Consider a state $\bm{B}$ where the local blockchains are given by
\begin{align}\label{eq: blocktree example N=7}
    B_i(t) = 
    \begin{cases}
        (1,2,3),  & \text{ if $i=1$}\\
        (1,2,4),  & \text{ if $i=2,5$}\\
        (1,2,3,6),  & \text{ if $i=6$}\\
        (1,2,3,5,8),  & \text{ if $i=3,4,7$}.\\
    \end{cases}
\end{align}

This state corresponds to the blocktree illustrated in \Cref{fig: blocktree}. The longest local blockchain is $(1,2,3,5,8)$, with height equal to $4$, which miners $3,4,7$ know about. The block with label $7$ is not present in any of the local blockchains. This must be because block $8$ was mined on top of block $5$ by one of the three miners $3,4,7$ and then propagated to the other two miners, one of which had previously mined block $7$. Block 7 is thus an example of an orphan block: it does not belong to any local blockchain, and will not be present in the long term blocktree. 

There are many possible transitions out of the blocktree state \eqref{eq: blocktree example N=7}. Every miner is continuously competing to mine the next block with block label $9$, and longer local blockchains are continuously being propagated to miners with shorter local blockchains.

\begin{figure}[H]
\centering

\begin{forest}
    for tree={
        grow'=north,
        draw,
        rectangle,
        minimum size=6mm,
        inner sep=1pt,
        l sep=8mm,
        s sep=20mm,
        edge={-},
        anchor=center
    }
    [1, name=n1, ultra thick, label=left:{\scriptsize {\it (genesis block)}}
      [2, name=n2, ultra thick, edge=ultra thick
        [4, name=n4, ultra thick, edge=ultra thick, label=right:{\scriptsize 2,5}]
        [3, name=n3, ultra thick, edge=ultra thick, label=right:{\scriptsize 1}
          [5, name=n5, ultra thick, edge=ultra thick
            [8, name=n8, ultra thick, edge=ultra thick, label=left:{\scriptsize {\it (block of maximum height)}}, label=right:{\scriptsize 3,4,7}]]
          [7, name=n7, thick, dotted, edge=dashed]
          [6, name=n6, ultra thick, edge=ultra thick, label=right:{\scriptsize 6}
            ] 
        ]
      ]
    ]
\end{forest}

\caption{The blocktree corresponding to \eqref{eq: blocktree example N=7}: The nodes represent blocks with their respective block labels; the numbers next to the nodes represent miners whose local blockchain ends in that node. The block with label $7$ is an orphan block and is no longer considered part of the blocktree.}
\label{fig: blocktree}
\end{figure}
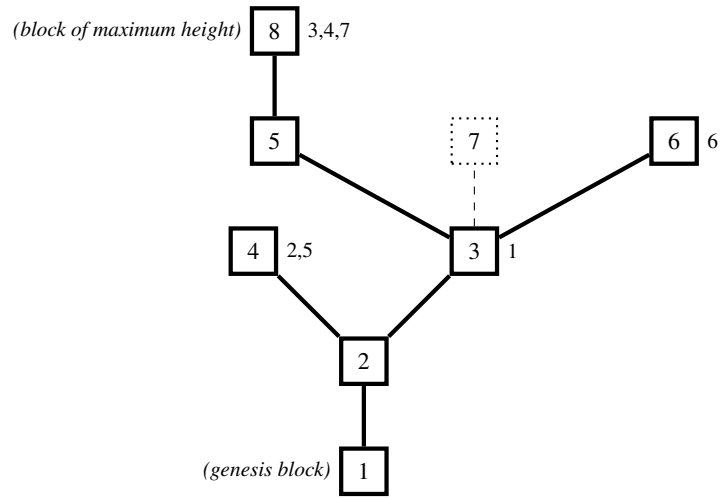

\subsection{Markov generator} We now formally define the Markov dynamics of our blockchain model. The positive integer $b := \max_{i \in [N]} b_i^{|B_i|}$ is the largest block label in $\bm{B}$, $\bm{B} \oplus_i b$ is the state produced by appending the block $b \in \N$ to the $i^{th}$ component of $\bm{B}$, and $\bm{B}^{i\to j}$ denotes the state obtained from $\bm{B}$ by setting the $i^{th}$ component equal to the $j^{th}$ component.

There are infinitely-many consensus states satisfying \eqref{eq: consensus Bi=Bj}. In this section, but not the rest of the paper where block heights will be important, we shall combine all of these states into a single consensus state $\bm{1}_N$. This doesn't affect the dynamics of the Markov chain because all blocks of height less than or equal to the highest labelled block in a consensus blocktree are guaranteed to remain forever. 

With initial condition $\bm{B}(0) = \bm{1}_N$ the blocktree state
\[
\bm{B}(t) := (B_1(t), B_2(t), \ldots, B_N(t)) \in \mathbb{B}^N, \quad t \geq 0, 
\]
evolves as a CTMC with transition rates given by 
\begin{align}
    &q^\mathrm{B}(\: \bm{B} \to \bm{B} \oplus_i (b+1) \:) := \lambda, &&i \in [N] \label{eq: B-process mining transition}\\
    &q^\mathrm{B}(\: \bm{B} \to \bm{B}^{i \to j} \:) := \frac{\omega}{N-1} \mathds{1}_{|B_j|>|B_i|}, &&i,j \in [N], \: i \neq j\label{eq: B-process propagation transition}
\end{align}
for any $\bm{B} \in \mathbb{B}^N$. These two transitions model mining and block propagation dynamics respectively. 

We implement combining all consensus states into a single state by also requiring that
\begin{equation}
q^\mathrm{B}(\: (B, \ldots, B, B_i, B, \ldots, B) \to b\bm{1}_N \:) := \omega \mathds{1}_{|B| > |B_i|},\quad  i \in [N].\label{eq: B-process consensus transition}
\end{equation}
This transition occurs when one of the miners not equal to $i$ communicates their local blockchain to miner $i$. 


\subsection{Consensus dynamics} Our first result is that the consensus state is positive recurrent for the blocktree process $\bm{B}(t)$.

\begin{prop}\label{prop: consensus positive reccurent}
     The consensus state \eqref{eq: consensus Bi=Bj} is positive recurrent for all $2 \leq N < \infty$.
\end{prop}

\begin{proof}
    For consensus to occur: (1) a new block must be mined that is higher than any other block, and (2) that block must be propagated to every other miner before any other block is mined at that same height. We will find a lower bound for the probability of this sequence of transitions occurring by considering a more restricted sequence of transitions: (1) a new highest block is mined, and (2) that block is propagated to all other miners before any other blocks are mined or propagated. In this restricted sequence of transitions we are ruling out mining and propagation of blocks other than the one that increased the maximum height of all the local blockchains.

    Consider the probability that the next transition is a mining event that generates a new block. In any given state there are $N$ mining transitions each of rate $\lambda$ and $2\binom{N}{2}$ communication events each of rate $\omega/(N-1)$ at most half of which can result in a block propagation transition. The probability that the next transition is mining a new `highest' block is
    \begin{equation*}
    P_{\text{highest}} \geq \frac{\lambda}{N\lambda + \frac{N\omega}{2}},
    \end{equation*}
    because there is at least one miner mining on the local blockchain of maximum height. If $n$ miners have already received this new highest block then there are $N-n$ miners still yet to receive it. Then the probability that the next transition will be propagation of that block is at least
    \begin{equation*}
    \frac{\frac{\omega}{N-1}n(N-n)}{N\lambda + \frac{N\omega}{2}}.
    \end{equation*}
    Therefore the probability that the next $N-1$ transitions after mining a new highest block result in propagation of that block to the rest of the network is
    \begin{equation*}
    P_{\text{prop}} \geq \prod_{n=1}^{N-1}\frac{\frac{\omega}{N-1}n(N-n)}{N\lambda + \frac{N\omega}{2}} = \frac{\left(\frac{\omega}{N-1}\right)^{N-1} (N-1)!^2}{\left( N\lambda + \frac{N\omega}{2} \right)^{N-1}}.
    \end{equation*}
    Regardless of what state the CTMC is in, there is at least a probability of $P_{\text{highest}}P_{\text{prop}}$ that a new highest block is mined and then propagated to all the other miners before any other blocks are mined. Thus the $N$-step transition probability to consensus from any state is
    \begin{equation*}
    P_{\text{cons}} \geq P_{\text{cons}}^* := \frac{\lambda}{N\lambda + \frac{N\omega}{2}} \frac{\left(\frac{\omega}{N-1}\right)^{N-1} (N-1)!^2}{\left( N\lambda + \frac{N\omega}{2} \right)^{N-1}}.
    \end{equation*}
    This lower bound does not depend on the state of the blockchain (no matter how fragmented the blockchain has become there is still a path to consensus in $N$ steps with a probability of at least $P_{\text{cons}}^*>0$ of occurring). If we think of the chain embedded at every $N$ steps, then it either reaches consensus after the next $N$ steps or it doesn't. Whether it does or not can be regarded as a Bernoulli trial that has a positive probability $P_{\text{cons}}$ of success. It follows that the expected time until the first success is finite and so the chain reaches consensus with probability 1 in finite expected time. Consensus is thus positive recurrent.
\end{proof}

\begin{remark}
    We can also bound the probability of transitioning into the consensus state from above. The probability of mining a new highest block as the next transition could be as high as $1$, but the probability of transmitting it successfully to all other miners prior to another block being mined is at most
    \begin{equation*}
        \prod_{n=1}^{N-1} \frac{n(N-n)\frac{\omega}{N-1}}{N\lambda + n(N-n)\frac{\omega}{N-1}} \leq \prod_{n=1}^{N-1} \frac{(N^2 / 4) \omega}{N(N-1)\lambda + (N^2 / 4) \omega}
        = \left( 1 + \frac{4(N-1)\lambda}{N\omega} \right)^{1-N}\!\!.
    \end{equation*}
    Therefore the consensus state can become very rare even for modest values of the ratio $\lambda/\omega$.
\end{remark}

\section{Height dynamics and the $N \to \infty$ limit}\label{sec: height dynamics} 

In this section and the rest of the paper, our main concern is the evolution of blockchain heights. For this reason, we will no longer collapse all consensus states into a single state via the transition \eqref{eq: B-process consensus transition}. Our Markov model will be defined entirely by the mining process with transition rates \eqref{eq: B-process mining transition} and the propagation process with transition rates \eqref{eq: B-process propagation transition}.

\subsection{Height function} To track the data of local blockchain heights we define the {\it height function} obtained through the mapping $\bm{B} \mapsto \bm{h}(\bm{B})$ defined by
\begin{align*}
    &\bm{h}(\bm{B}) := (|B_1|-1, \ldots, |B_N|-1) \in \N_{0}^N, \quad \bm{B} \in \mathbb{B}^N.
\end{align*}
The height function inherits the initial condition $\bm{h}(0) = \bm{0}_N$ and evolves as a CTMC with transition rates given by
\begin{align}
    &q^\mathrm{h}(\: \bm{h} \to \bm{h} + \bm{\e}_i \:) := \lambda,   &&i \in [N]; \label{eq: h-process mining transition}\\
    &q^\mathrm{h}(\: \bm{h} \to \bm{h} + (h_j-h_i) \bm{\e}_i \:) := \frac{\omega}{N-1} \mathds{1}_{h_i<h_j},   &&i,j \in [N], \: i \neq j, \label{eq: h-process propagation transition}
\end{align}
for any $\bm{h} \in \N_{0}^n$.

\subsubsection{Transition probability} Let $\mathbb{P}_t(\bm{h})$ denote the transition probability of evolving into $\bm{h} \in \N_{0}^N$ after some time $t$ from the initial condition $\bm{h}(0) = \bm{0}_N$. Then one can show that the transition probability function is symmetric
\begin{align*}
&\mathbb{P}_t(h_1, \ldots, h_N) = \mathbb{P}_t(h_{\sigma(1)}, \ldots, h_{\sigma(N)}),   &&\sigma \in S_N, \: (\bm{h},t) \in \N_{0}^N \times \R_{\geq 0},
\end{align*}
where $S_N$ is the set of permutations of $[N]$. Using this symmetry, \eqref{eq: h-process mining transition} and \eqref{eq: h-process propagation transition}, we see that the transition probability function satisfies the system of Kolmogorov forward differential equations
\begin{multline*}
    \frac{d}{dt}\mathbb{P}_t(\bm{h}) =  \lambda \sum_{i=1}^N \left( \mathbb{P}_t(h_1,\ldots,h_i-1,\ldots,h_N) - \mathbb{P}_t(h_1,\ldots,h_N) \right)\\
    + \frac{\omega}{N-1} \sum_{i<j} \left( 2\mathds{1}_{h_i=h_j} \times \left( \sum_{k_i=0}^{h_j-1} \mathbb{P}_t(h_1, \ldots, k_i, \ldots, h_j, \ldots, h_N) \right) -  \mathds{1}_{h_i \neq h_j} \mathbb{P}_t(h_1,\ldots,h_N) \right),
\end{multline*}
subject to the initial and boundary conditions
\begin{align*}
    &\mathbb{P}_0(\bm{h}) = \prod_{i=1}^N\delta_{h_i,0}, &&\mathbb{P}_t(h_1,\ldots,-1_i,\ldots,h_N)=0, \quad i \in [N], \: h_{j \neq i} \in \N_{0}.
\end{align*}

\subsection{Occupation representation}

In order to track aggregated quantities, such as the number of miners at a certain height, it is convenient to define the {\it occupation function} obtained through the mapping $\bm{h} \mapsto \bm{\eta}(\bm{h})$ defined by
\begin{align*}
    \bm{\eta}(\bm{h}) := (\eta_0, \eta_1, \eta_2, \ldots) \in \Omega_N, \quad \eta_i := \sum_{k=1}^N \delta_{i, h_k},  \quad \bm{h} \in \N_{0}^N,
\end{align*}
where $\Omega_N$ is the space of infinite sequences of non-negative integers with normalisation $N$,
\begin{align*}
    \Omega_N := \left\{ \bm{\eta} \in \N_{0}^\infty \: : \: \sum_{i=0}^\infty \eta_i = N \right \}.
\end{align*}
We interpret the components $\eta_i$ to denote the number of miners whose local blockchain has height equal to $i$. Then the occupation function inherits the initial condition $\bm{\eta}(0) = N \bm{\e}_0$ and, from \eqref{eq: h-process mining transition}, \eqref{eq: h-process propagation transition}, $\bm{\eta}(t)$ evolves as a CTMC on $\Omega_N$ with transition rates given by
\begin{align}
    &q^\eta(\bm{\eta} \to \bm{\eta} + \bm{\e}_{i+1} - \bm{\e}_i) := \lambda \eta_i,   &&i \geq 0; \label{eq: eta-process mining transition}\\
    &q^\eta(\bm{\eta} \to \bm{\eta} + \bm{\e}_j - \bm{\e}_i) := \frac{\omega}{N-1} \eta_i \eta_j \mathds{1}_{i<j},   &&i,j \geq 0,\label{eq: eta-process propagation transition}
\end{align}
for any $\bm{\eta} \in \Omega_N$.

For a given state $\bm{\eta}$, let $\bm{\eta}_{max}= \sup\{i:\eta_i>0\}$, which must be finite because $\sum_{i=0}^\infty \eta_i = N$. In the Markov chain $\bm{\eta}(t)$, the rate of increase of $\bm{\eta}_{max}$ is bounded above by a Poisson process with rate $N\lambda$. This Markov chain thus cannot (with positive probability) pass through infinitely-many states in a finite time. That is, it is not explosive. 

\subsubsection{Transition probability} Let $\mathbb{P}_t(\bm{\eta})$ denote the probability that  $\bm{\eta}(t) =  \bm{\eta} \in \Omega_N$ with initial condition $\bm{\eta}(0) = N \bm{\e}_0$. Then, from \eqref{eq: eta-process mining transition}, \eqref{eq: eta-process propagation transition}, $\mathbb{P}_t(\bm{\eta})$ satisfies the Kolmogorov forward differential equations
\begin{multline*}
    \frac{d}{dt}\mathbb{P}_t(\bm{\eta}) =  \lambda \left( \sum_{0 \leq i} (\eta_i+1) \mathbb{P}_t(\bm{\eta} + \bm{\e}_i - \bm{\e}_{i+1}) - N \mathbb{P}_t (\bm{\eta}) \right)\\
    + \frac{\omega}{N-1} \left( \sum_{0 \leq i<j} (\eta_i+1)(\eta_j-1) \mathbb{P}_t(\bm{\eta} + \bm{\e}_i - \bm{\e}_j) -  \left( \sum_{0 \leq i < j} \eta_i \eta_j \right) \mathbb{P}_t(\bm{\eta}) \right),
\end{multline*}
subject to the initial and boundary conditions
\begin{align*}
    &\mathbb{P}_0(\bm{\eta}) = \delta_{\eta_0,N},\\
    &\mathbb{P}_t(0,\ldots,0,(N+1)_i,0,\ldots,0,(-1)_j,0,\ldots) = 0, \quad 0 \leq i < j,\\
    & \mathbb{P}_t(\eta_0, \eta_1,\ldots, (-1)_i,\ldots) = 0, \quad 0 \leq i.
\end{align*}

\begin{remark}
    A spatially continuous analogue of the $\bm{\eta}(t)$-process defined above was described by Baryshnikov and Stolyar \cite{Stolyar_2025}.  They considered a system of $N$ particles moving on the real line $\R$ in which particles make independent jumps at points of a Poisson process with parameter $\lambda$ whose sizes are independent and identically distributed random variables $Z > 0$ with $\mathbb{E}(Z) =1$ and distribution function $J$. In addition, the process has "synchronization" jumps, occurring at unit rate, at which a particle joins a randomly-chosen other particle if the latter happens to be ahead of it.

    Considering the model of \cite{Stolyar_2025} in our context, the random variables $Z$ are equal to 1 with probability 1 and we don't necessarily scale time so that $\omega=1$. The Laplace transform of the jump distribution is $L(s)=e^{-s}$ and the parameter 
    \begin{align*}
        \alpha = \liminf_{x \to \infty} \frac{-\log(1-J(x))}{x}
    \end{align*}
    of \cite{Stolyar_2025} is equal to $\infty$. From equation (7.3) of \cite{Stolyar_2025}, the minimum speed of a travelling wave is given by the minimiser of the function 
    \begin{align*}
        V(\zeta) = \frac{1}{\zeta}\left[e^{\zeta} - \lambda +1\right] 
    \end{align*}
    over $\zeta>0$, which occurs at the point $\zeta^*$ that satisfies the equation $\lambda e^{\zeta^*} (\zeta^*-1) = 1-\lambda$. In our discussion of the speed and shape of the limiting profile in Section 5, we reproduce this result in equation \eqref{eq:gamma_0equation}.
    
     We recommend reading Stolyar's preceding line of work \cite{Stolyar_1996, Stolyar_2023_a, Stolyar_2023_b} and the references therein.

\end{remark}

\subsection{Miner density profiles}

We now have the governing equations to study a large miner limit for our model. Define the {\it finite-miner density profile} obtained through the mapping $\bm{\eta} \mapsto \bm{\theta}^{N}(\bm{\eta})$ defined by
\begin{align*}
    \bm{\theta}^{N}(\bm{\eta}) := (\theta^N_0, \theta^N_1, \theta^N_2, \ldots), \quad \theta^N_i := \frac{\eta_i}{N},  \quad \bm{\eta} \in \Omega_N.
\end{align*}
Note that each $\theta^N_i \in [0,1] \cap \tfrac{1}{N}\N_0 $ and $\sum\nolimits_{i \geq 0} \theta^N_i = 1$. Then the state
\begin{equation*}
\bm{\theta}^N(t) := (\theta^N_0(t), \theta^N_1(t), \ldots),
\end{equation*}
is a random variable that can be thought of as the empirical
probability mass function of local blockchain heights at time $t$, inheriting the initial condition $\bm{\theta}^N(0) = \bm{\e}_0$ and, from \eqref{eq: eta-process mining transition} and \eqref{eq: eta-process propagation transition}, evolving as a CTMC with transition rates given by
\begin{align}
    &q^{\theta}(\bm{\theta}^N \to \bm{\theta}^N + \tfrac{1}{N}\bm{\e}_{i+1} - \tfrac{1}{N}\bm{\e}_i) = \lambda N \theta^N_i,   &&i \geq 0,\label{dd1}\\
    &q^\theta(\bm{\theta}^N \to \bm{\theta}^N + \tfrac{1}{N}\bm{\e}_j - \tfrac{1}{N}\bm{\e}_i) = \frac{\omega}{N-1}  N^2 \theta^N_i  \theta^N_j \mathds{1}_{i<j},   &&i,j \geq 0. \label{dd2}
\end{align}

The CTMC with transition rates \eqref{dd1} and \eqref{dd2} is a density-dependent Markov chain. Using this observation, we can invoke the well-known results of Kurtz, \cite{Kurtz_1980} (for models with finitely many compartments), and Barbour and Luczak, \cite{BarbourLuczak_2008} (for the extension to countably many compartments), to establish a Functional Law of Large Numbers. The details are given in the next section. 

\subsection{Limiting differential equations} 
We aim to approximate, for large $N$, the evolution of the finite $N$-miner $\bm{\theta}^N$-process, defined by \eqref{dd1} and \eqref{dd2}. A natural candidate  approximation is given by the solution to the "average drift" system of ODEs
\begin{equation}
    \left \{
    \begin{aligned}\label{eq: avg drift ODEs}
        \frac{d}{dt}\bm{\tilde{\theta}}^N(t) &= \sum_{i=0}^{\infty} \lambda \tilde{\theta}^N_i(t)(\bm{\e}_{i+1} - \bm{\e}_i) + \sum_{i=0}^{\infty}\sum_{j=i+1}^{\infty} \frac{\omega N}{N-1} \tilde{\theta}^N_j(t)\tilde{\theta}^N_i(t)(\bm{\e}_j - \bm{\e}_i),\\
        \bm{\tilde{\theta}}^N(0) &= \bm{\e}_0.
\end{aligned}
    \right.
\end{equation}


The result below is a special case of \cite[Theorem 3.1]{BarbourLuczak_2008}, due to Barbour-Luczak 2008. With the $\ell_1$-norm $\| \bm{x} \|_1 := \sum_{i \geq 0} |x_i|$ and the $\ell_{11}$-norm defined as $\| \bm{x} \|_{11} := \sum_{i \geq 0} (i+1)|x_i|,$ it quantifies the $\ell_1$-error in the approximation of the stochastic $\bm{\theta}^N$-process with the deterministic solution $\bm{\tilde{\theta}}^N$ of the system of ODEs \eqref{eq: avg drift ODEs}.
\begin{thm}[Barbour-Luczac 2008]\label{thm: FLLN Barbour-Luczac 2008} Suppose there exists a solution $\bm{\tilde{\theta}}^N$ to the system of ODEs \eqref{eq: avg drift ODEs}, on $t \geq 0$, such that $\bm{\tilde{\theta}}^N \in \ell_{11}$. Then, for all $T>0$, there exists a constant $K(T)>0$ such that, as $N \to \infty$,
    \begin{align*}
        \bm{\mathrm{P}} \left[ N^{-1}\sup_{0 \leq t \leq T} \|  \bm{\eta}^N(t) -  N \bm{\tilde{\theta}}^N(t) \|_1 > K(T) N^{-1/2} (\log N)^{3/2} \right] = \order(N^{-1/2}).
    \end{align*}
\end{thm}
\begin{proof}
    Recall the $\bm{\eta}$-process defined by \eqref{eq: eta-process mining transition} and \eqref{eq: eta-process propagation transition}. This corresponds to a parameter specialisation of the process defined in \cite[Section 2]{BarbourLuczak_2008} which, in their notation, is
    \begin{align*}
        X^N(t) := \bm{\eta}^N(t), \quad \bar{\alpha}_{i,j} := \lambda \mathds{1}_{j = i+1}, \quad
    \alpha_{i,j}(N^{-1}\bm{\eta}^N) := \frac{\omega}{N-1} \eta_j^N \mathds{1}_{i<j}, \quad i \geq 0.
    \end{align*}
    We note that the $\bm{\eta}$-process was shown earlier to be non-explosive and that conditions (2.7)-(2.9) of \cite{BarbourLuczak_2008} are satisfied by a simple check. The result then follows from \cite[Theorem 3.1]{BarbourLuczak_2008}.

\end{proof}
\begin{remark}
    We have not proved the hypothesis of \Cref{thm: FLLN Barbour-Luczac 2008} that there exists a solution to \eqref{eq: avg drift ODEs} which belongs to $\ell_{11}$. This can be demonstrated by application of the Picard-Lindel\"of theorem, which we omit for brevity. However, in \Cref{lm: varphi_n properties} below, we do prove existence and uniqueness for the ODEs characterising the limiting empirical distribution functions $\varphi_n(t)$, defined in \Cref{prop:BC}, which forms the primary focus of this paper.
\end{remark}

\begin{cor}[ODEs for the limiting density profile]\label{prop: limiting equations}
    The components of the limiting miner density profile $\bm{\theta}(t) := \lim_{N \to \infty} \bm{\theta}^{N}(t)$ solve the system of ODEs
    \begin{equation}\label{eq:Ninfty_setODEs}
        \frac{d}{dt}\bm{\theta}(t) = \sum_{i=0}^{\infty} \lambda \theta_i(t)(\bm{\e}_{i+1} - \bm{\e}_i) + \sum_{i=0}^{\infty}\sum_{j=i+1}^{\infty} \omega \theta_j(t)\theta_i(t)(\bm{\e}_j - \bm{\e}_i), 
    \end{equation}
    with the initial condition $\bm{\theta}(0) = \bm{\e}_0$ and normalisation $\sum_{i \geq 0} \theta_i(t) = 1$.
\end{cor}
\begin{proof}
    This follows from \Cref{thm: FLLN Barbour-Luczac 2008} after passing the system of ODEs \eqref{eq: avg drift ODEs} to the limit.
\end{proof}

\begin{remark}
    The components of \eqref{eq:Ninfty_setODEs} are given by
    \begin{align}
         \frac{d}{dt}\theta_0(t) &= -\lambda \theta_0(t) -\omega \theta_0(t)(1-\theta_0(t)), \label{eq:theta1} \\
         \frac{d}{dt}\theta_i(t) &= -\lambda\left( \theta_i(t) - \theta_{i-1}(t) \right)
         -\omega \theta_i(t) \left( \sum_{j=i+1}^{\infty}\theta_j(t) - \sum_{j=0}^{i-1}\theta_j(t) \right), \quad i \geq 1,\label{eq:theta2}
    \end{align}
    each with initial condition $\theta_i(0) = \delta_{i,0}$ for $i \geq 0$. 
\end{remark}

\begin{prop}[ODEs for the limiting distribution function]\label{prop:BC}
    Define the limiting empirical distribution function
\begin{equation}\label{eq:BCcdf}
    \varphi_n(t) := \sum_{i=0}^{n} \theta_i(t), \quad  n \in \N_0.
\end{equation}
    Then the components of the sequence $( \varphi_n )_{n \geq 0}$ solve the system of ODEs 
\begin{equation}\label{eq:BC}
    \left\{
    \begin{aligned}
    \dot{\varphi}_n(t) &= -\lambda(\varphi_n - \varphi_{n-1}) - \omega \varphi_n(1-\varphi_n), \quad &&n \geq 0,\\
    \varphi_n(0) &= 1, \quad &&n \geq 0,
    \end{aligned}
    \right.
\end{equation}
with boundary condition $\varphi_{-1} \equiv 0$. Moreover, the $n=0$ solution is given by
\begin{align}\label{eq:BCinitsol}
    \varphi_0(t) = \frac{\lambda+\omega}{\omega+\lambda \mathrm{e}^{(\lambda+\omega)t}}, \quad t \geq 0.
\end{align}
\end{prop}
\begin{proof}
    See \Cref{proof: BC}
\end{proof}

\begin{remark} 
    The function \eqref{eq:BCcdf} represents the proportion of miners, whose height is $n$ or less at time $t$, in the limiting system with an arbitrarily large number of miners. In the case $\omega=0$, \eqref{eq:BC} has well known solutions 
    \begin{equation*}
        \varphi_n|_{\omega=0}(t) = \e^{-\lambda t} \sum_{k=0}^{n}\frac{(\lambda t)^k}{k!},
    \end{equation*}
    corresponding to the Poisson distribution function, that is $\varphi_n|_{\omega=0}(t) = P(n \leq N_t)$ for $N_t \sim \text{Poisson}(\lambda t)$. The resulting asymptotic behaviour is that the front travels with speed $\lambda$ but spreads out with rate $\sqrt{t}$ (and so no steady state travelling wave is ever reached). In the case $\lambda=0$, the hierarchy collapses and the problem is degenerate. This paper concerns the case $\lambda,\omega>0$. Numerical solutions of \eqref{eq:BC} predict that the sequence $(\varphi_n)_{n \geq 0}$ converges to a non-trivial limiting profile within a moving frame, which we prove in Section~\ref{sec: Convergence to limiting profile}. This implies that the blocktree $\bm{B}(t)$-process with any positive rate of communication $\omega>0$, however small, remains in consensus within a bounded distance of its maximum height for all time.
\end{remark}

\section{Convergence to limiting profile}\label{sec: Convergence to limiting profile}

In this section we prove that the sequence of solutions $(\varphi_n)_{n \geq 0}$, of the system of ODEs \eqref{eq:BC} defined in \Cref{prop:BC}, converges uniformly to a nontrivial limiting profile within an appropriate comoving frame. This convergence result is made precise and formally stated in \Cref{thm: u_n convergence}.

\begin{lm}[Properties of $\varphi_n$]\label{lm: varphi_n properties}
     The distribution functions $(\varphi_n )_{n \geq 0}$ are smooth on $\R_{\geq 0}$ and are the unique set of solutions to \eqref{eq:BC}. Moreover, for all $n \geq 0$ and $t > 0$, 
    \begin{align*}
        \mathrm{(i)}\ 0 < \varphi_n(t) < 1; \quad \mathrm{(ii)}\ \varphi_n(t) - \varphi_{n-1}(t) > 0; \quad \mathrm{(iii)} \ \dot{\varphi}_n(t) < 0;
    \end{align*}
    and, for all $n \geq 0$, 
    \begin{align*}
        \mathrm{(iv)} \ \varphi_n(t) \sim \left( \frac{\lambda^{n-1}(\lambda+\omega)}{n!}t^n + \sum_{k=0}^{n-1} \frac{1}{k!}(\lambda t)^k \right) \mathrm{e}^{-(\lambda+\omega)t} \quad \text{as $t \to \infty$}.
    \end{align*}
\end{lm}
\begin{proof}
    See \Cref{proof: BCsolproperties}.
\end{proof}

\begin{defn}[Frame fixing]
    To eliminate the translation drift of the profiles $(\varphi_n)_{n \geq 0}$, we perform a frame fixing by pinning the level sets at the origin. For $p \in [0,1]$ the $p$-level set of a given $\varphi_n$ is $\{t:\varphi_n(t) = p\}$. From property (iii) of \Cref{lm: varphi_n properties}, it is clear that the level set of each profile $\varphi_n$ has a single element - hence we will refer to the level set as the hitting time. This defines a comoving frame with respect to the chosen level. 
    
    Fix a level $p \in (0,1)$. From \Cref{lm: varphi_n properties} we define the hitting time and inter-hitting time
    \begin{align}
        t_n &:= \inf\{ t>0 \ | \ \varphi_n(t) = p \} = \varphi_n^{-1}(p),  &&n \geq 0, \label{eq: t_n}\\
        \tau_n &:= t_n - t_{n-1},  &&n \geq 1,   \label{eq: tau_n}
    \end{align}
    and assume $\tau_0:=t_0$. It follows that $\tau_n>0$ for all $n \geq 1$ due to the ordering
    \begin{align*}
        0 < t_0 < t_1 < t_2 < \cdots, 
    \end{align*}
    which follows from property (ii) of \Cref{lm: varphi_n properties}. To track the profile $\varphi_n$, in the comoving frame with respect to level $p$, we define for all $n \geq 0$ the frame shifted profile
    \begin{equation}\label{eq: u_n definition}
        u_n(t) :=
        \begin{cases}
            \varphi_n(t+t_n),  & \text{ if $t \geq -t_n$},\\
            1,  & \text{ if $t < -t_n$},\\
        \end{cases}
    \end{equation}
    which, from \eqref{eq: t_n}, satisfies the normalisation condition $u_n(0) = p$. Furthermore, due to the initial condition $\varphi_n(0)=1$ in \eqref{eq:BC}, $u_n \in \Cf(\R)$ for all $n \geq 0$.
\end{defn}

\begin{thm}\label{thm: u_n convergence}
    There exists $u \in \Cf^1(\R)$ and $\tau \geq 0$ such that
    \begin{align}
    &\lim_{n \to \infty} \sup_{t \in \R} |u_n(t) - u(t)| = 0,\label{eq convthm: lim_nsup_tR|u_n-u|=0}\\
    &\lim_{n \to \infty} \tau_n = \tau, \label{eq convthm: lim tau_n=tau}
    \end{align}
    where $u$ solves the boundary value problem
    \begin{equation}\label{eq convthm: u limit BVP}
    \left\{
    \begin{aligned}
        &\dot{u}(t) = -\lambda(u(t) - u(t+\tau)) - \omega u(t)(1-u(t)),  &&t \in \R, \\
        &\lim_{t \to -\infty} u(t) = 1, \quad \lim_{t \to \infty} u(t) = 0, \quad u(0)=p.
    \end{aligned}
    \right.
    \end{equation}
    Furthermore, $u \in \Cf^\infty(\R)$ and satisfies
    \begin{align}\label{eq convthm: u limit boundedness and monotonicity}
        0<u(t)<1 \quad \text{and} \quad \dot{u}(t) < 0, \quad \text{for all } t \in \R,
    \end{align}
    the limit $\tau$ is independent of $p \in (0,1)$ and is explicitly given by
    \begin{align}
        \tau &= \max_{\alpha \geq \omega} \frac{1}{\alpha} \log \left( \frac{\alpha-\omega}{\lambda} + 1 \right). \label{eq convthm: tau=maxtau}
    \end{align}
\end{thm}

The remainder of this section is devoted to the proof of this convergence result, \Cref{thm: u_n convergence}. 

\subsection{Outline of proof of \Cref{thm: u_n convergence}}\label{subsection: outline of proof of convergence theorem} Our proof makes use of four preliminary results:

\begin{enumerate}
    \item In \Cref{lm: u_n properties}, we transfer properties of the profile $\varphi_n$ to the frame shifted profile $u_n$. We also see that $u_n \in \Cf^1(\R)$ if and only if  $n \geq 1$. Hence at times we may only consider the tail sequence $(u_n)_{n \geq 1}$ which is sufficient.
    \item In \Cref{lm: hittingtimes rel compact}, we prove that the sequence of inter-hitting times $(\tau_n)_{n \geq 1}$ is bounded, with an upper bound depending on level $p$; and we prove that the sequence of hitting times $(t_n)_{n \geq 0}$ approaches $\infty$ in the limit of large $n$.
    \item In \Cref{lm: u_n pointwise monotone}, we prove that the frame shifted profile sequence $(u_n)_{n \geq 0}$ is pointwise monotone on $\R$. In fact, we prove a stronger result that gives strict monotonicity on the positive half-line, and eventually strict monotonicity on the negative half-line.
    \item In \Cref{lm: uniform estimates}, we prove multiple uniform estimates for the sequence $(u_n)_{n \geq 0}$ and the sequence of derivatives $(\dot{u}_n)_{n \geq 1}$.
\end{enumerate}

\noindent
A sketch of the proof of \Cref{thm: u_n convergence} is as follows:

\noindent
To prove \eqref{eq convthm: lim_nsup_tR|u_n-u|=0}, we use the pointwise monotone result of \Cref{lm: u_n pointwise monotone} and the uniform bound \eqref{ineq: uniform bounds} in \Cref{lm: uniform estimates} which gives a unique pointwise limit of $(u_n)_{n \geq 1}$ on $\R$. Using the estimate \eqref{ineq: uniform bounds}, in conjunction with the Arzel\`a--Ascoli theorem and a diagonal argument, we obtain a uniformly convergent subsequence of $(u_n)_{n \geq 1}$ on all compact intervals, and combining with the unique pointwise limit we obtain uniform convergence of the entire sequence on compact intervals. We then use the uniform tail estimates \eqref{eq: u_n tail estimate t=+infty} and \eqref{eq: u_n tail estimate t=-infty} of \Cref{lm: uniform estimates} to extend to global uniform convergence on $\R$. This gives the result \eqref{eq convthm: lim_nsup_tR|u_n-u|=0} for $u \in \Cf(\R)$.

To prove \eqref{eq convthm: lim tau_n=tau}, we evaluate the ODE in the shifted frame \eqref{eq: u_n odes} at the origin; which then gives $u_{n-1}(\tau_n)$ equal to $\tfrac{1}{\lambda}\dot{u}_n(0)$ plus a constant in terms of the level $p \in (0,1)$ and $\lambda, \omega>0$. Using the pointwise monotone result of \Cref{lm: u_n pointwise monotone} and the uniform bound \eqref{ineq: uniform bounds} we infer that $\dot{u}_n(0)$ is monotone and bounded and so converges, hence also $(u_{n-1}(\tau_n))_{n \geq 1}$ converges. We then use \Cref{lm: hittingtimes rel compact} to obtain two convergent subsequences of $(\tau_n)_{n \geq 1}$ with assumed different limits. We then use the unique pointwise limit $u \in \Cf(\R)$ and the uniform derivative bound \eqref{ineq: uniform dt upper bound} of \Cref{lm: uniform estimates} to show uniqueness of these limits which gives the result \eqref{eq convthm: lim tau_n=tau}.

To prove \eqref{eq convthm: u limit BVP} and \eqref{eq convthm: u limit boundedness and monotonicity}, we use the result \eqref{eq convthm: lim_nsup_tR|u_n-u|=0} and the uniform bounds in \Cref{lm: uniform estimates} to pass the frame shifted ODE \eqref{eq: u_n odes} to the limit. Using the completeness of $\Cf^1$ on compact intervals and the derivative uniform tail estimates \eqref{eq: u_n^(j) tail estimate t=+-infty} in \Cref{lm: uniform estimates}, we obtain that $u \in \Cf^1(\R)$ which also completes the proof of \eqref{eq convthm: lim_nsup_tR|u_n-u|=0}. The fact that $u \in \Cf^\infty(\R)$ is obtained by a simple induction argument.

To prove \eqref{eq convthm: tau=maxtau}, we bound solutions to \eqref{eq:BC} by an exponential function of $n,t$ and $\tau$ and then evaluate them at the hitting times $t_n$. The exponential function is constructed so as to satisfy the linearised problem of \eqref{eq:BC} in the vicinity of the fixed point $\varphi_n = 1$. Using our inequality on the times $t_n$ we obtain that $\tau$ is not less than the expression in \eqref{eq convthm: tau=maxtau}. We then obtain equality by the monotone condition \eqref{eq convthm: u limit boundedness and monotonicity}.

\subsection{Preliminary estimates}\label{subsection: preliminary estimates} We present the technical tools which we use to prove \Cref{thm: u_n convergence}. The following result, \Cref{lm: u_n properties}, transfers properties of the profile $\varphi_n$ to the frame shifted profile $u_n$. This forms the basis for proving the technical estimates which follow.

\begin{cor}[\textit{Properties of $u_n$}]\label{lm: u_n properties}
    For all $n \geq 0$, the shifted function $u_n$ is smooth on $\R_{> -t_n}$, and the sequence $(u_n)_{n \geq 0}$ is the unique set of solutions to the system of ODEs
    \begin{equation}
    \left \{
    \begin{aligned}\label{eq: u_n odes}
        &\dot{u}_n(t) = -\lambda (u_n(t) - u_{n-1}(t + \tau_n)) - \omega u_n(t) (1 - u_n(t)),  &&n \geq 0, \: t \geq -t_n,\\
        &u_n(-t_n) = 1, &&n \geq 0,
    \end{aligned}
    \right.
    \end{equation}
    with boundary condition $u_{-1}(t) = \mathds{1}_{t<-t_0}$. Moreover, for all $n \geq 0$ and $t > -t_n$, 
    \begin{align*}
        \mathrm{(i)} \: 0 < u_n(t) < 1; \quad \mathrm{(ii)} \: u_n(t) - u_{n-1}(t+\tau_n) > 0; \quad \mathrm{(iii)} \: \dot{u}_n(t) < 0;
    \end{align*}
    and, for all $k \geq 0$,
    \begin{align*}
        \mathrm{(iv)} \: u_n \in \Cf^k(\R) \quad \text{if and only if} \quad n \geq k.
    \end{align*}
\end{cor}
\begin{proof}
    Smoothness, uniqueness, and properties (i),(ii), and (iii) follows from \Cref{lm: varphi_n properties} and \eqref{eq: u_n definition}. To prove property (iv), one can verify from \eqref{eq:BC} that $\varphi_n^{(k)}(0) = 0$ if $1 \leq k \leq n$, and $\varphi_n^{(n+1)}(0) = -\lambda^{n+1} \neq 0$. Then combining with \eqref{eq: u_n definition} the result follows.
\end{proof}

\begin{figure}[H]\label{fig: u_n solution in comoving frame}
    \centering
    \begin{tikzpicture}
        \begin{axis}[
            width=0.9\textwidth,
            height=0.4\textwidth,
            xmin=-8, xmax=16,
            ymin=0, ymax=1.05,
            axis y line=middle,
            axis x line=bottom,
            xtick={-8,-4,0},
            xticklabels={$-t_n$,$-t_{n-1}$,0},
            ytick={0.5, 1},
            yticklabels={$p$,1},
            yticklabel style={
            anchor=west,
            xshift=3pt,
            yshift=3pt
        },
            tick label style={font=\small},
            samples=300,
            domain=-8:16,
            smooth,
            thick,
            grid=none,
            clip=false,
        ]
        
            \addplot[
            domain=-4:16,
            blue!0!black,
            line width=1.2pt,
            solid
            ] {
                (1+1.67835*((x+4)/4))*exp(-1.67835*((x+4)/4))
            };
            \node[
            anchor=south,
            font=\normalsize
            ] at (axis cs:-4,1)
                {$u_{n-1}$};

            \addplot[
            domain=-8:16,
            red!0!black,
            line width=1.2pt,
            solid
            ] {
                (1+1.67835*((x+8)/8))*exp(-1.67835*((x+8)/8))
            };
            \node[
            anchor=south,
            font=\normalsize
            ] at (axis cs:-8,1)
                {$u_n$};

            \addplot[
            dotted,
            black!70,
            line width=0.9pt
            ] coordinates {(-8,1) (0,1)};

            \addplot[
            blue!0!black,
            line width=1.2pt,
            solid
            ] coordinates {(-8,1) (-4,1)};

            \addplot[
            dotted,
            black!70,
            line width=0.9pt
            ] coordinates {(-8,0) (-8,1)};
            
            \addplot[
            dotted,
            black!70,
            line width=0.9pt
            ] coordinates {(-4,0) (-4,1)};

            \draw[
                decorate,
                decoration={brace, amplitude=4pt, mirror},
                thick
            ] (axis cs:-8,-0.18) -- (axis cs:-4,-0.18);
            
            \node[
                anchor=north,
                font=\small
            ] at (axis cs:-6,-0.22) {$\tau_n := t_n-t_{n-1}$};

                
        \end{axis}
    \end{tikzpicture}
    \caption{Solutions $u_n$ and $u_{n-1}$ in the comoving frame with respect to a level $p \in (0,1)$. The fact that they intersect only at $t=0$, on the interval $t>-t_n$, is a direct consequence of \Cref{lm: u_n pointwise monotone}.}
    \label{fig:logistic}
\end{figure}
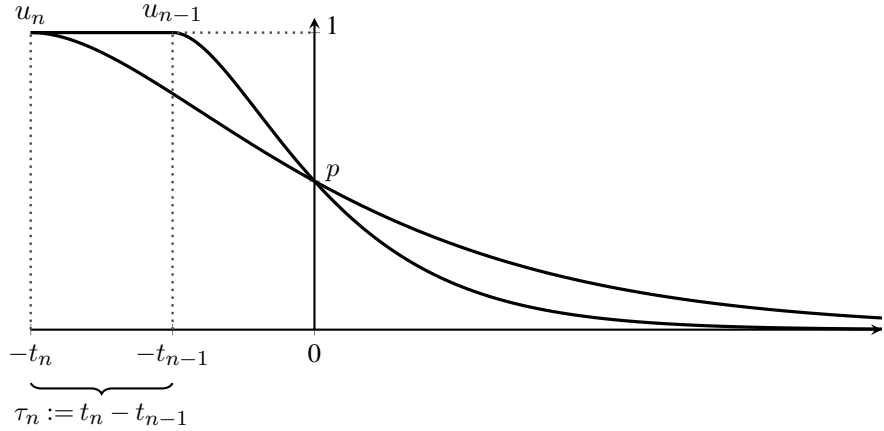

\begin{lm}\label{lm: hittingtimes rel compact} \quad

    \begin{itemize}
        \item [ (i)] For every $p \in (0,1)$, there exists $T_p>0$ such that $0 \leq \tau_n \leq T_p, \quad n \geq 1$.
    \item [(ii)] $t_n \to \infty$ as $n \to \infty$.
    \end{itemize}
\end{lm}
\begin{proof}
    
    To prove (i), consider the initial value problem 
    \begin{align}\label{eq: T_p comparison ode}
        \dot{f} = -\lambda (f - p) - \omega f (1 - f), \quad f(t_{n-1})=1, \quad n \geq 1.
    \end{align}
    From \Cref{lm: varphi_n properties} and \eqref{eq: t_n} we deduce that
    \begin{align*}
         \varphi_n(t_{n-1}) \in (p,1), \quad n \geq 1, \quad \text{and} \quad
         \varphi_{n-1}(t) < p, \quad n \geq 1, \: t > t_{n-1}.
    \end{align*}
    Then comparing \eqref{eq:BC} with \eqref{eq: T_p comparison ode}, it follows by the comparison principle for ODEs that
     \begin{align*}
         \varphi_n(t) < f(t), \quad n \geq 1, \: t \geq t_{n-1}.
     \end{align*}
     From \eqref{eq: T_p comparison ode} we infer $\dot{f}(t_{n-1}) = -\lambda (1-p)< 0$ independent of $n \geq 1$, and $\dot{f}<0$ for all $f \in [p,1]$. Hence $f$ reaches $p$ in some finite time $T_p < \infty$ independent of $n$. Then using the comparison inequality above and $\varphi_n(t_n)=p$, it follows that $t_n \leq T_p + t_{n-1}$ for all $n \geq 1$.

     To prove (ii), consider the system of ODEs 
     \begin{align*}
         \dot{f}_n(t) &= (\omega-\lambda)f_n(t) + \lambda f_{n-1}(t), \quad f_n(0)=0, \: f_{-1}(t)=1,
         \end{align*}
         and its solution
         \begin{align*}
         f_n(t) &= \frac{\lambda^{n+1}}{n!}\e^{(\omega-\lambda)t} \int_0^t \e^{-(\omega-\lambda)s} s^n \: ds.
     \end{align*}
     Define $g_n(t):=1-\varphi_n(t)$, then from \Cref{lm: u_n properties} it follows that $\dot{g}_n(t) \leq \dot{f}_n(t)$ on $n,t \geq 0$. Then using the comparison principle of ODEs, we compute
     \begin{align}\label{ineq: 0<1-varphi<f}
         0 \leq g_n(t) \leq f_n(t) \leq \frac{(\lambda t)^{n+1}}{(n+1)!} \e^{\omega t}, \quad n \geq 0, \: t \geq 0,
     \end{align}
     where we used $\e^{(\omega-\lambda)(t-s)} \leq \e^{\omega t}$ on $0 \leq s \leq t$ to bound and evaluate the integral. It then follows that $\lim_{n \to \infty} |\varphi_n(t)| = 1$ for any finite $ t\geq 0$. Since $\varphi_n(t_n)=p$, it follows that the sequence $(t_n)_{n \geq 0}$ cannot remain bounded.
\end{proof}

The following technical estimate, \Cref{lm: u_n pointwise monotone}, is a pointwise monotonicity property for the frame shifted profile sequence $(u_n)_{n \geq 0}$. The estimate \eqref{ineq: w_n>0} concerns the positive half-line and says that, for each $t>0$, the sequence $(u_n(t))_{n \geq 0}$ is strictly increasing. The estimate \eqref{ineq: w_n<0} concerns the negative half-line and says that, for each $t<0$, the sequence $(u_n(t))_{n \geq 0}$ is eventually strictly decreasing. This follows from utilising \Cref{lm: hittingtimes rel compact} which tells us that, for each $t<0$, there exists $N_t \geq 1$ such that $-t_n < t$ for all $n \geq N_t$, that is the tail sequence $(u_n(t))_{n \geq N_t}$ is strictly decreasing on $-t_{N_t} < t < 0$.

\begin{lm}\label{lm: u_n pointwise monotone}
    For all $n \geq 1$, 
    \begin{align}
        0 < u_{n-1}(t) < u_n(t) < p& \quad \text{if $t > 0$},\label{ineq: w_n>0}\\
        p < u_n(t) < u_{n-1}(t) \leq 1& \quad \text{if $-t_n < t < 0$}.\label{ineq: w_n<0}
    \end{align}
\end{lm}
\begin{proof}
    From \Cref{lm: u_n properties}, for all $n\geq0$, it follows  that
    \begin{align*}
        0<u_n(t)<p \: \text{ if $t>0$}, \quad \text{and} \quad p<u_n(t)\leq1 \: \text{ if $t<0$}.
    \end{align*}
    Define $w_n(t) := u_n(t) - u_{n-1}(t)$. Then it suffices to show, for all $n \geq 1$, that
    \begin{align*}
        w_n(t)>0 \: \text{ if $t>0$}, \quad \text{and} \quad w_n(t)<0 \: \text{ if $-t_n<t<0$}.
    \end{align*}
    From \eqref{eq: u_n definition} and the initial condition $\varphi_{n-1}(0)=1$, it follows that
    \begin{align*}
        w_n(t) = 
        \begin{cases}
            \varphi_n(t+t_n)-1, & \text{ if $-t_n<t \leq -t_{n-1}$},\\
            \varphi_n(t+t_n) - \varphi_{n-1}(t+t_{n-1}), & \text{ if $t>-t_{n-1}$},
        \end{cases}
    \end{align*}
    and combining with property (i) of \Cref{lm: varphi_n properties}, it follows that
    \begin{align}\label{ineq: w_n<0 small -t}
        w_n(t)<0 \quad \text{if $-t_n<t \leq -t_{n-1}$}.
    \end{align}
    From property (iv) of \Cref{lm: varphi_n properties}, we compute
    \begin{align*}
        w_n(t) \sim \frac{\lambda^{n-2}(\lambda+\omega)}{(n-1)!} \left( \frac{\lambda}{n}\e^{-(\lambda+\omega)\tau_n}(t+t_n)^n - (t+t_{n-1})^{n-1} \right) \e^{-(\lambda+\omega)(t+t_{n-1})}
    \end{align*}
    as $t \to \infty$,
    which is positive for all sufficiently large $t$ and hence
    \begin{align}\label{ineq: w_n>0 large +t}
        w_n(t) > 0 \quad \text{for all sufficiently large $t$}.
    \end{align}
    From \eqref{ineq: w_n<0 small -t} and \eqref{ineq: w_n>0 large +t}, it follows that there must be some $t > -t_{n-1}$ at which $w_n(t)=0$. From \eqref{eq: t_n} and \eqref{eq: u_n definition}, $w_n(0)=0$ for all $n \geq 1$. Therefore, it suffices to show $w_n(0)=0$ is the unique zero of $w_n$ on $t > -t_{n-1}$ for all $n \geq 1$.

    It is convenient for the induction step below to define the shifted function
    \begin{align}\label{eq: h_ns}
        h_{n,s}(t) := u_n(t+s-\tau_n)-u_{n-1}(t), \quad n \geq 1, \: s>0,\: t > -t_{n-1}-s,
    \end{align}
    where $\tau_n$ is defined by \eqref{eq: tau_n}. The introduction of $s$ will be used to prove uniqueness of a zero for any $s>0$. When $s=\tau_n$ this reduces to $h_{n, \tau_n}(t) = w_n(t)$ and, for all $s > 0$, we have
    \begin{align}
        h_{n,s} < 0& \quad \text{if $-t_{n-1}-s<t \leq -t_{n-1}$},\label{ineq: h_ns<0}\\
        h_{n,s} > 0& \quad \text{for all sufficiently large $t$},\label{ineq: h_ns>0}
    \end{align}
    which follows by the same reasoning which led to \eqref{ineq: w_n<0 small -t} and \eqref{ineq: w_n>0 large +t}. This gives existence of a zero of $h_{n,s}$ on the interval $t>-t_{n-1}$. Hence it suffices to show, for all $s > 0$, the shifted function $h_{n,s}$ cannot have more than one zero on the interval $t>-t_{n-1}$.
    
    Differentiating \eqref{eq: h_ns} on the interval $t > -t_{n-1}$, and using \eqref{eq: u_n odes}, we obtain
    \begin{align}\label{eq: dt h_ns}
        \dot{h}_{n,s} = \alpha_{n,s}(t) h_{n,s}(t) + \lambda h_{n-1,s}(t+\tau_{n-1}), \quad t > -t_{n-1}, 
    \end{align}
    where $\alpha_{n,s}(t) := -\lambda-\omega(1-u_n(t+s-\tau_n)-u_{n-1}(t))$. If $n=1$, then \eqref{eq: dt h_ns} simplifies to
    \begin{align*}
        \dot{h}_{1,s}(t) = -\alpha_{1,s} h_{1,s}(t) + \lambda \varphi_0(t+s+t_0), \quad t > -t_0,
    \end{align*}
    where we used \eqref{eq: u_n definition}, and the boundary condition $u_{-1}(t) = \mathds{1}_{t<-t_0}$ in \Cref{lm: u_n properties}. Suppose $x^{1,s} > -t_0$ is a zero of $h_{1,s}$. Then $h_{1,s}(x^{1,s})=0$ and we obtain
    \begin{align*}
        \dot{h}_{1,s}(x^{1,s}) =  \lambda \varphi_0(x^{1,s}+s+t_0),
    \end{align*}
    which is positive for all $s>0$ and $x^{1,s} > -t_0$, by \Cref{lm: varphi_n properties}. It then follows that, for every $s>0$, $\dot{h}_{1,s}$ is positive at all zeros of ${h}_{1,s}$ on the interval $t>-t_0$. Combining with \eqref{ineq: h_ns<0} and \eqref{ineq: h_ns>0}, the zero $x^{1,s}>-t_0$ is unique.
    Choosing $s=\tau_1$ gives $w_1(0)=0$ as the unique zero of $w_1$ on $t>-t_0$.

    We proceed by induction for $n \geq 2$. Assume that, for all $s > 0$, there exists a unique $x^{n-1,s} > -t_{n-2}$ such that $h_{n-1,s}(x^{n-1,s})=0$ and $\dot{h}_{n-1,s}(x^{n-1,s})>0$. Suppose for a contradiction that $x_1^{n,s} < x_2^{n,s}$ are the first and second zeros of $h_{n,s}$ with nonzero derivative, on the interval $t > -t_{n-1}$, respectively. It then follows that
    \begin{align*}
        h_{n,s}(x_1^{n,s}) = h_{n,s}(x_2^{n,s})=0, \quad \dot{h}_{n,s}(x_1^{n,s})>0, \quad \dot{h}_{n,s}(x_2^{n,s})&<0,
    \end{align*}
    where the inequalities follow from \eqref{ineq: h_ns<0}, that is, $h_{n,s}$ must cross the first zero from negative to positive and the second zero from positive to negative. Combining with \eqref{eq: dt h_ns}, we obtain
    \begin{align}
        \dot{h}_{n,s}(x_1^{n,s}) &= \lambda h_{n-1,s}(x_1^{n,s}+\tau_{n-1})>0 \label{eq: dt h_ns(x_ns)},\\
        \dot{h}_{n,s}(x_2^{n,s}) &= \lambda h_{n-1,s}(x_2^{n,s}+\tau_{n-1})<0\label{eq: dt h_ns(y_ns)}.
    \end{align}
    From the induction hypothesis, $h_{n-1,s}$ has a unique zero at $x^{n-1,s}>-t_{n-2}$ at which it crosses from negative to positive. Combining with \eqref{eq: dt h_ns(x_ns)} and \eqref{eq: dt h_ns(y_ns)}, 
    \begin{align*}
        x_2^{n,s} + \tau_{n-1} < x^{n-1,s} < x_1^{n,s}+\tau_{n-1},
    \end{align*}
    which gives a contradiction since $x_1^{n,s}<x_2^{n,s}$ were assumed to be the first and second zeros of $h_{n,s}$ with nonzero derivative on the interval $t>t_{n-1}$. Therefore, for all $s > 0$, $h_{n,s}$ has only one zero with nonzero derivative on the interval $t>t_{n-1}$, whose existence follows from \eqref{ineq: h_ns<0} and \eqref{ineq: h_ns>0}. As above, we may fix $s = \tau_n$ which gives $h_{n,\tau_n}(t) = w_n(t)$, and conclude that $w_n(0)=0$ is the unique zero for $w_n$ on $t > -t_{n-1}$.
\end{proof}

\begin{lm}\label{lm: uniform estimates}
    \begin{align}
        \sup_{n \geq 0} \sup_{t \in \R} u_n(t) &\leq 1, \quad \sup_{n \geq 1} \sup_{t \in \R} |\dot{u}_n(t)| \leq \lambda+\frac{\omega}{4},\label{ineq: uniform bounds}\\
         \sup_{n \geq 0} u_n(t) &= \order(\e^{-\omega t}) \quad \text{as $t \to \infty$}, \label{eq: u_n tail estimate t=+infty}\\
        \sup_{n \geq 0} 1-u_n(t) &= \order(\e^{\omega t}) \quad \text{as $t \to -\infty$}, \label{eq: u_n tail estimate t=-infty}\\
        \sup_{n \geq 1} |\dot{u}_n(t)| &= \order(\e^{-\omega |t|}) \quad \text{as $t \to \pm \infty$}, \label{eq: u_n^(j) tail estimate t=+-infty}
    \end{align}
    and, for all $r>0$, there exists $N_r,C_r>0$ such that
    \begin{align}
    \sup_{n \geq N_r} \sup_{|t| \leq r} \dot{u}_n(t) \leq - C_r.\label{ineq: uniform dt upper bound}
    \end{align}
\end{lm}
\begin{proof}
    See \Cref{proof: uniform estimates}.
\end{proof}

\subsection{Proof of \Cref{thm: u_n convergence}}\label{subsection: proof of convergence theorem} We now have the technical machinery to prove \Cref{thm: u_n convergence}. Since our results concern the limiting behaviour of the sequence of shifted profiles $(u_n)_{n \geq 0}$, it suffices to consider the tail sequence $(u_n)_{n \geq 1}$. This allows utilisation of all the estimates in \Cref{lm: uniform estimates}.  We prove each line in the order as stated:

 \subsubsection{Proof of \eqref{eq convthm: lim_nsup_tR|u_n-u|=0}} From \Cref{lm: u_n pointwise monotone}, for every $t \in \R$, the sequence $(u_n(t))_{n \geq 1}$ is monotone and bounded. Hence the pointwise limit exists
 \begin{align}\label{eq: u(t) pointwise limit}
     u(t) := \lim_{n \to \infty} u_n(t), \quad t \in \R.
 \end{align}
 
From \Cref{lm: uniform estimates}, specifically \eqref{ineq: uniform bounds}, the sequence $(u_n)_{n \geq 1}$ is uniformly bounded and uniformly equicontinuous on $\R$. Applying the Arzel\`a--Ascoli theorem, every subsequence of $(u_n)_{n \geq 1}$ admits a further subsequence converging uniformly on a given compact subset of $\R$. A standard diagonal argument iterates this scheme and produces a single subsequence which converges uniformly on all compact subsets of $\R$; for example, letting $r=1$, there exists a subsequence $A_r \subseteq \N$ such that $(u_a)_{a \in A_r}$ converges uniformly on $[-r,r] \subset \R$. Then pass to a further subsequence $A_{r+1} \subseteq A_r$ such that $(u_a)_{a \in A_{r+1}}$ converges uniformly on $[-r-1,r+1] \subset \R$. Proceeding iteratively, the limit of this procedure gives a subsequence $A_\infty \subseteq \N$ such that $(u_a)_{a \in A_\infty}$ converges uniformly on all compact subsets of $\R$.

The pointwise limit \eqref{eq: u(t) pointwise limit} identifies a unique limit for all uniformly convergent subsequences of $(u_n)_{n \geq 1}$. Therefore, for all $r>0$, 
\begin{align}\label{eq: compact uniform convergence}
    \lim_{n \to \infty} \sup_{|t| \leq r}|u_n(t)-u(t)| = 0,
\end{align}
where $u$ is the pointwise limit defined in \eqref{eq: u(t) pointwise limit}. This gives us local uniform convergence on $\R$. We now use the uniform tail estimates in \Cref{lm: uniform estimates} to extend \eqref{eq: compact uniform convergence} to global uniform convergence on $\R$. From \eqref{eq: u_n tail estimate t=+infty} and \eqref{eq: u_n tail estimate t=-infty}, there exists $C,R>0$ such that
\begin{align*}
    \sup_{t>R}|u_n(t)-u(t)| &\leq \sup_{t>R}|u_n(t)| + \sup_{t>R}|u(t)| < 2C\e^{-\omega R}, \quad n \geq 1,\\
    \sup_{t<-R}|u_n(t)-u(t)| &\leq \sup_{t<-R}|1-u_n(t)| + \sup_{t<-R}|1-u(t)| < 2C\e^{-\omega R}, \quad n \geq 1.
\end{align*}
Let $\varepsilon>0$ and choose $R>0$ sufficiently large such that $2C\e^{-\omega R}<\tfrac{\varepsilon}{3}$. From \eqref{eq: compact uniform convergence}, choose $N>0$ sufficiently large such for all $n \geq N$, $\sup_{|t| \leq R}|u_n(t)-u(t)| < \tfrac{\varepsilon}{3}$. Combining the three estimates and the triangle inequality, we compute
\begin{align*}
    \sup_{t \in \R}|u_n(t)-u(t)| \leq \sup_{t<-R}|u_n(t)-u(t)| + \sup_{|t| \leq R}|u_n(t)-u(t)| + \sup_{t>R}|u_n(t)-u(t)| < \varepsilon,
\end{align*}
for sufficiently large $R>0$ and all sufficiently large $n$. The result \eqref{eq convthm: lim_nsup_tR|u_n-u|=0} then follows. Moreover, since $(u_n)_{n \geq 1} \subset \Cf^1(\R)$ (due to property (iv) of \Cref{lm: u_n properties}) and convergence is uniform, the limit \eqref{eq: u(t) pointwise limit} is $u \in \Cf(\R)$. As a byproduct of the proof of \eqref{eq convthm: u limit BVP} below, we will obtain that $u \in \Cf^1(\R)$.

\subsubsection{Proof of \eqref{eq convthm: lim tau_n=tau}} Let $q_n := \dot{u}_n(0)$, $n \geq 1$. Evaluating \eqref{eq: u_n odes} at $t=0$, we define
    \begin{align}\label{eq: s_n := unm1 at t=0}
        y_n := u_{n-1}(\tau_n) = \tfrac{1}{\lambda}(q_n + \lambda p + \omega p(1-p)), \quad n \geq 1.
    \end{align}
    From \Cref{lm: u_n pointwise monotone} and \Cref{lm: uniform estimates}, it follows that $(q_n)_{n \geq 1}$ is monotone and bounded, and therefore converges. The monotonicity follows from a simple argument. For example, for the right-sided derivative,
    \begin{align*}
        \dot{u}_n(0) = \lim_{t \to 0^+} \frac{u_n(t)-u_n(0)}{t} \geq \lim_{t \to 0^+} \frac{u_{n-1}(t)-u_{n-1}(0)}{t} = \dot{u}_{n-1}(0),
    \end{align*}
    where we used \eqref{ineq: w_n>0} of \Cref{lm: u_n pointwise monotone} and $u_n(0)=u_{n-1}(0)=p$ to obtain the inequality. The left-sided derivative argument is similar but instead uses \eqref{ineq: w_n<0} of \Cref{lm: u_n pointwise monotone}. Combining with \eqref{eq: s_n := unm1 at t=0}, it then follows that $(y_n)_{n \geq 1}$ converges to some limit depending on $p$, hence
    \begin{align}\label{eq: sbar limit}
        y^p := \lim_{n \to \infty} y_n.
    \end{align}
    Recall the pointwise limit \eqref{eq: u(t) pointwise limit}. Then from \eqref{ineq: uniform dt upper bound}, it satisfies
    \begin{align}\label{ineq: pointwise limit ordering}
        p > u(s) > u(t) > 0 \quad \text{for all } 0 < s < t.
    \end{align}
    From \Cref{lm: hittingtimes rel compact} and the Bolzano--Weierstrass theorem, it follows that $(\tau_n)_{n \geq 1}$ is relatively compact in $[0,T_p]$ for some $T_p>0$. Let $M_1, M_2 \subset \N_1$ be two subsequences such that 
     \begin{align*}
         \tau_m \longrightarrow \tau^{M_1} \in [0,T_p] \: \text{ along $M_1$} \quad \text{and} \quad \tau_m \longrightarrow \tau^{M_2} \in [0,T_p] \: \text{ along $M_2$}.
     \end{align*}
     Then passing \eqref{eq: s_n := unm1 at t=0} to the limit along $M_1$ and $M_2$, and using \eqref{eq: sbar limit} and \eqref{eq: u(t) pointwise limit}, we obtain
     \begin{align*}
         y^p = u(\tau^{M_1}) \quad \text{and} \quad y^p = u(\tau^{M_2}).
     \end{align*}
    Combining with \eqref{ineq: pointwise limit ordering} we may invert $u$ and therefore conclude $\tau^{M_1} = \tau^{M_2}$. Hence all subsequential limits are unique and the result \eqref{eq convthm: lim tau_n=tau} then follows.

\subsubsection{Proof of \eqref{eq convthm: u limit BVP}}
Let $u$ and $\tau$ be defined by \eqref{eq convthm: lim_nsup_tR|u_n-u|=0} and \eqref{eq convthm: lim tau_n=tau}, and define 
\begin{align}\label{eq: v(t) = u ode}
    v(t) := - \lambda(u(t) - u(t+\tau)) - \omega u(t) (1-u(t)), \quad t \in \R.
\end{align}
Then $v \in \Cf(\R)$ since $u \in \Cf(\R)$, which was established in the proof of \eqref{eq convthm: lim_nsup_tR|u_n-u|=0}. Consider the half-line $t \geq -t_n$. Then the ODE \eqref{eq: u_n odes} is defined on this half-line, and using the triangle inequality, we compute
\begin{align}\label{ineq: |dot u_n - v| for t>=-t_n}
    |\dot{u}_n(t)-v(t)| \leq &\lambda |u_n(t)-u(t)| + \lambda|u_{n-1}(t+\tau_n)-u(t+\tau)| \\
    &+ \omega|1-u_n(t)-u(t)|\cdot|u_n(t)-u(t)| \nonumber\\
    \leq &(\lambda+\omega) |u_n(t)-u(t)| + \lambda|u_{n-1}(t+\tau_n)-u(t+\tau)| \nonumber\\
    \leq &(\lambda+\omega)|u_n(t)-u(t)| + \lambda|u_{n-1}(t+\tau)-u(t+\tau)| \nonumber\\
    &+ \lambda|u_{n-1}(t+\tau_n)-u_{n-1}(t+\tau)| \nonumber\\
    \leq &(\lambda+\omega)|u_n(t)-u(t)| + \lambda|u_{n-1}(t+\tau)-u(t+\tau)| \nonumber \\
    &+ \lambda \left(\lambda+\frac{\omega}{4}\right) |\tau_n - \tau|, \quad t \geq -t_n, \: n \geq 2, \nonumber
\end{align}
where we used the uniform bound \eqref{ineq: uniform bounds} of \Cref{lm: uniform estimates} for the second inequality; addition and subtraction of $u_{n-1}(t+\tau)$ and the triangle inequality for the third inequality; and the mean value theorem and the uniform bound \eqref{ineq: uniform bounds} the the fourth inequality (valid for $n \geq 2$).

Now consider the other half-line $t<-t_n$. Then $\dot{u}_n \equiv 0$ on this half-line, due to \eqref{eq: u_n definition}. From the uniform estimate \eqref{eq: u_n tail estimate t=-infty}, there exists $C_1,R_1>0$ such that $1-u(t) \leq C_1\e^{\omega t}$ for all $t \leq -R_1$; and $C_2,R_2>0$ such that $1-u(t+\tau) \leq C_2\e^{\omega (t+\tau)}$ for all $t \leq -R_2-\tau$. Then take $C:= \max\{C_1, C_2\}$ and $R_\tau:= \max \{ R_1, R_2+\tau\}$. From \Cref{lm: hittingtimes rel compact}, there exists $N_{R_\tau}>0$ such that $-t_n<-R_\tau$ for all $n \geq N_{R_\tau}$. Using these facts we compute from \eqref{eq: v(t) = u ode} that
\begin{align}\label{ineq: sup|dot u_n - v| for t<-t_n}
    \sup_{t<-t_n}|\dot{u}_n(t)-v(t)|
    &\leq (\lambda+\omega) \sup_{t<-t_n}|1 - u(t)| + \lambda \sup_{t<-t_n}|1-u(t+\tau)| \\
    &\leq (\lambda+\omega) \sup_{t<-t_n} C\e^{\omega t} + \lambda \sup_{t<-t_n}C\e^{\omega(t+\tau)} \nonumber\\
    &= C(\lambda+\omega+\lambda \e^{\omega \tau})\e^{-\omega t_n}, \quad n \geq N_{R_\tau}, \nonumber
\end{align}
where we used the addition and subtraction of $1$ inside the $\lambda$-factor of \eqref{eq: v(t) = u ode}, and the uniform bound \eqref{ineq: uniform bounds} inside the $\omega$-factor of \eqref{eq: v(t) = u ode}, to obtain the first inequality.

Then combining the half-line inequalities \eqref{ineq: |dot u_n - v| for t>=-t_n} and \eqref{ineq: sup|dot u_n - v| for t<-t_n}, we compute
\begin{align*}
    \sup_{t \in \R}|\dot{u}_n(t) - v(t)| &\leq \sup_{t < -t_n}|\dot{u}_n(t) - v(t)| + \sup_{t \geq -t_n}|\dot{u}_n(t) - v(t)|\\
    &\leq C(\lambda+\omega+\lambda \e^{\omega \tau})\e^{-\omega t_n} + \lambda \left(\lambda+\frac{\omega}{4}\right) |\tau_n - \tau|\\
    &+ (\lambda+\omega)\sup_{t \geq -t_n}|u_n(t)-u(t)| + \lambda\sup_{t \geq -t_n}|u_{n-1}(t+\tau)-u(t+\tau)|,
\end{align*}
for $n \geq \max\{N_{R_\tau},2\}$. Passing to the limit $n \to \infty$ we have $t_n \to \infty$, due to \Cref{lm: hittingtimes rel compact}. Hence the first term on the right-hand-side vanishes and, using \eqref{eq convthm: lim tau_n=tau}, the second term also vanishes. The last two terms are bounded above by extending the supremum from $t \geq -t_n$ to $t \in \R$ and, using \eqref{eq convthm: lim_nsup_tR|u_n-u|=0}, the last two terms also vanish. We therefore obtain
 \begin{align*}
    \lim_{n \to \infty} \sup_{t \in \R} |\dot{u}_n(t) - v(t)| = 0.
\end{align*}
 The standard $\Cf^1$ completeness argument \cite{Rudin_1976} gives $\dot{u}(t) = v(t)$ on compact subsets of $\R$. Combining with the uniform tail estimate \eqref{eq: u_n^(j) tail estimate t=+-infty}, and the same argument that extended local uniform convergence \eqref{eq: compact uniform convergence} to global uniform convergence \eqref{eq convthm: lim_nsup_tR|u_n-u|=0}, it follows that
\begin{align*}
    \lim_{n \to \infty} \sup_{t \in \R}|\dot{u}_n(t)-\dot{u}(t)| = 0.
\end{align*}
An immediate corollary is that $u$ solves the ODE \eqref{eq convthm: u limit BVP} and $u \in \Cf^1(\R)$. We verify the boundary conditions in \eqref{eq convthm: u limit BVP}: The right boundary condition follows from the uniform tail estimate \eqref{eq: u_n tail estimate t=+infty}, and the left boundary condition follows from the uniform tail estimate \eqref{eq: u_n tail estimate t=-infty}. Passing $u_n(0)=p$ to the limit gives the normalisation $u(0)=p$.

\subsubsection{Proof of \eqref{eq convthm: u limit boundedness and monotonicity} and $u \in \Cf^\infty(\R)$}
\begin{itemize}
    \item Boundedness follows from \Cref{lm: u_n pointwise monotone}: Passing \eqref{ineq: w_n>0} and \eqref{ineq: w_n<0} to the limit gives
    \begin{align*}
        0<u(t):=\lim_{n \to \infty} u_n(t) \leq p \quad \text{if } t>0, \quad \text{and} \quad p\leq u(t):=\lim_{n \to \infty} u_n(t) < 1 \quad \text{if } t<0,
    \end{align*}
    then combining with $u(0)=p$ gives the result.
    \item Strict monotonicity follows from \eqref{ineq: uniform dt upper bound}: For any $t \in \R$, choose $r_t>|t|$, then \eqref{ineq: uniform dt upper bound} gives existence of $C_{r_t}>0$ such that $\dot{u}(t) := \lim_{n \to \infty} \dot{u}_n(t) \leq -C_{r_t}$.
    \item Smoothness follows from the ODE of \eqref{eq convthm: u limit BVP}: We already know that $u \in \Cf^1(\R)$ and solves the boundary value problem \eqref{eq convthm: u limit BVP}. Repeated differentiation of \eqref{eq convthm: u limit BVP} gives $u^{(k+1)}$ in terms of lower order derivatives. The result is then obtained by an induction argument.
\end{itemize}

\subsubsection{Proof of \eqref{eq convthm: tau=maxtau} and $\tau$ independent of $p$}

We first prove independence. Consider two shifted frames with limiting profiles $u^{p_1}$ and $u^{p_2}$ with respect to levels $p_1,p_2 \in (0,1)$ such that $p_1<p_2$. The ODE of \eqref{eq convthm: u limit BVP} is translation invariant, hence there exists $s^{p_1,p_2}>0$ such that $u^{p_2}(t + s^{p_1,p_2}) = u^{p_1}(t), \: t \in \R$. Plugging this into the $u^{p_1}$ ODE of \eqref{eq convthm: u limit BVP} evaluated at $t=-s^{p_1,p_2}$ and subtracting the $u^{p_2}$ ODE of \eqref{eq convthm: u limit BVP} evaluated at $t=0$, we obtain $u^{p_2}(\tau^{p_2}) = u^{p_2}(\tau^{p_1})$. Since $u^{p_2}$ is strictly decreasing on $\R$, due to the uniform bound \eqref{ineq: uniform dt upper bound}, it follows that $\tau^{p_1}=\tau^{p_2}$.

We now prove \eqref{eq convthm: tau=maxtau}. Let $\tau$ be defined by \eqref{eq convthm: lim tau_n=tau} and set $\tau^*$ equal to the right-hand-side of \eqref{eq convthm: tau=maxtau}. We first prove $\tau \geq \tau^*$. If we linearise the limiting ODE \eqref{eq convthm: u limit BVP} near the fixed point $u(-\infty)=1$ we obtain the linear differential equation
\begin{align}\label{eq: limit u linear problem}
    \dot{\tilde{u}}(t) = -\lambda (\tilde{u}(t) - \tilde{u}(t+\tau)) - \omega(1 - \tilde{u}(t)), \quad t \in \R.
\end{align}
With $\tilde{u}(-\infty)=1$ this equation has solution $\tilde{u}(t) = 1 - A \e^{\alpha t}$ with characteristic equation
\begin{align}\label{eq: tau alpha characteristic equation}
    \alpha = \lambda (\e^{\alpha \tau} - 1) + \omega, \quad \text{Re}(\alpha) \geq \omega.
\end{align}
The linearly predicted inter-hitting time is then 
\begin{align}\label{eq: tau(alpha)}
    \tau(\alpha) := \frac{1}{\alpha} \log \left( \frac{\alpha-\omega}{\lambda} + 1 \right), \quad \text{Re}(\alpha) \geq \omega.
\end{align}
which is positive on $\alpha>\omega$, satisfies $\tau(\omega) = \tau(\infty) = 0$ and has a unique global maximum defined by \eqref{eq convthm: tau=maxtau}.

We now construct an exponential bounding function. Recall the system of ODEs \eqref{eq:BC} and define the change of variable $g_n = 1 - \varphi_n$. We obtain
\begin{equation}\label{eq: g_n ODE}
    \left\{
    \begin{aligned}
        &\dot{g}_n = \lambda(g_{n-1} - g_n) + \omega g_n(1-g_n), && n \geq 0, \: t \geq 0,\\
        &g_n(0) = 0, &&n \geq 0.
    \end{aligned}
    \right.
\end{equation}
Define the exponential bounding function $G_n(t) := \e^{\alpha (t - n \tau(\alpha))}$ with $\tau(\alpha)$ defined by \eqref{eq: tau(alpha)}. Differentiating $G_n$ and using \eqref{eq: tau alpha characteristic equation}, we obtain the linear differential equation 
\begin{equation}\label{eq: G_n ODE}
    \left\{
    \begin{aligned}
        &\dot{G}_n = \lambda(G_{n-1} - G_n) + \omega G_n, && n \geq 0, \: t \geq 0,\\
        &G_n(0) = \e^{-\alpha n \tau(\alpha)}, &&n \geq 0.
    \end{aligned}
    \right.
\end{equation}
Then by comparing \eqref{eq: g_n ODE} with \eqref{eq: G_n ODE} and restricting to $\alpha \in \R_{\geq \omega}$, it follows by the comparison principle for ODEs that $g_n(t) \leq  G_n(t)$ for all $n,t \geq 0$. Evaluating this inequality on the sequence of hitting times $(t_n)_{n \geq 0}$, we obtain $1-p \leq \e^{\alpha (t_n - n \tau(\alpha))}$, which is equivalently
\begin{align}\label{eq: t_n >= n tau(alpha) + Cp}
    t_n \geq n \tau(\alpha) + \alpha^{-1} \ln (1-p), \quad n \geq 0, \: \alpha \in \R_{\geq \omega}.
\end{align}
We thus conclude that $t_n/n \ge \tau^*$ where $\tau^*$ is that maximum of $\tau(\alpha)$ over all $\alpha\ge \omega$. Because $t_n = t_0 + \sum_{i=1}^n \tau_i$ by virtue of the definitions \eqref{eq: t_n} and \eqref{eq: tau_n}, we conclude that $t_n / n \to \tau$ with $\tau$ defined by \eqref{eq convthm: lim tau_n=tau}, and therefore $\tau \geq \tau^*$.

We now prove $\tau \leq \tau^*$. The characteristic equation \eqref{eq: tau alpha characteristic equation} does not have real roots if $\tau > \tau^*$. For complex solutions of $\tau$ we would obtain an exponential mode with oscillating amplitude of the linear problem \eqref{eq: limit u linear problem}. This leads to a contradiction since the solution of the nonlinear problem \eqref{eq convthm: u limit BVP} satisfies the monotone property \eqref{eq convthm: u limit boundedness and monotonicity}, and hence cannot be asymptotic to an oscillating profile.

\section{Speed and shape of limiting profile}\label{sec: speed and shape of limiting profile}

In the previous section we showed that the sequence of solutions $(\varphi_n)_{n \geq 0}$, of the system of ODEs \eqref{eq:BC} in \Cref{prop:BC}, converges uniformly to a nontrivial limiting profile in an appropriate comoving frame. This was the result of \Cref{thm: u_n convergence}. In this section we study the travelling wave speed and the shape of the limiting profile, also called the travelling wave shape.

A travelling wave solution of \eqref{eq:BC} with speed $v>0$ is a solution of the form $\Phi_v(z) := \varphi_x(t)$, where $z := x-vt$, which solves the boundary value problem
\begin{equation}\label{eq:BCtw}
    \left \{
    \begin{aligned}
        &v \frac{d}{dz}\Phi(z) = \lambda(\Phi(z)-\Phi(z-1)) + \omega\Phi(z)(1-\Phi(z)), \quad z \in \R,\\
        &\lim_{z \to -\infty} \Phi(z) = 0, \quad \lim_{z \to \infty} \Phi(z) = 1.
    \end{aligned}
    \right.
\end{equation}
\begin{remark}
    The boundary value problem \eqref{eq:BCtw} is equivalent to the boundary value problem \eqref{eq convthm: u limit BVP}, in \Cref{thm: u_n convergence}, under the change of variable:
    \begin{align}\label{eq: change of variable: advODE-delayODE}
        v := \frac{1}{\tau}, \quad z := -\frac{t}{\tau}, \quad \Phi(z) := u(t).
    \end{align}
\end{remark}

\subsection{Minimum speed} \label{ss: min_speed_selection} We first determine the minimum possible speed $v$ of a travelling wave solution of \eqref{eq:BC}, which is a solution $\Phi_v$ that satisfies the boundary value problem \eqref{eq:BCtw}. We follow a minimum speed argument described in \cite{BrunetDerrida_1997}, where it was applied in the Fisher-KPP class.

\begin{prop}
If the travelling wave shape $\Phi_v$ is asymptotically exponential for large $z$, then the minimum possible speed is
\begin{align}\label{eq:v0=lambdaExpgamma0}
    v_0 := \lambda \mathrm{e}^{\gamma_0},
\end{align}
with $\gamma_0$ defined implicitly by
\begin{align}\label{eq:gamma_0equation}
    \lambda \mathrm{e}^{\gamma_0}(1-\gamma_0) = \lambda- \omega.
\end{align}
\end{prop}

\begin{proof}
If $\Phi_v$ is asymptotically exponential, then there exist constants $A,\gamma>0$ such that
\begin{align}\label{eq:twsol_decay1}
    \Phi_v(z) \sim 1-A\mathrm{e}^{-\gamma z} \quad \text{as $z \to \infty$}.
\end{align}
The linearisation of \eqref{eq:BCtw} around the stable fixed point $\Phi=1$ is
\begin{align}\label{eq:BCtw_linear}
    v \frac{d}{dz}\Phi^{\text{lin}}(z) = \lambda(\Phi^{\text{lin}}(z)-\Phi^{\text{lin}}(z-1)) + \omega(1-\Phi^{\text{lin}}(z)).
\end{align}
Substituting the asymptotic form \eqref{eq:twsol_decay1} into  \eqref{eq:BCtw_linear} gives the characteristic equation
\begin{align}\label{eq:v(gamma)}
    v(\gamma) := \frac{\lambda( \mathrm{e}^\gamma-1) + \omega}{\gamma}, \quad \gamma>0.
\end{align}
This establishes dependence of the travelling wave speed $v$ on the exponential decay rate $\gamma$ of the right tail of the travelling wave shape $\Phi_v$.

On the domain $\gamma>0$, the function \eqref{eq:v(gamma)} is strictly convex, positive, and blows up in the left and right boundary limits $\gamma \to 0$ and $\gamma \to \infty$, respectively. Therefore, there exists a unique minimum speed $v_0 := v(\gamma_0)>0$ for some unique $\gamma_0>0$; and for each speed $v>v_0$ there exists unique $\gamma_1,\gamma_2>0$ such that $v(\gamma_1)=v(\gamma_2)$ with $0<\gamma_1<\gamma_0<\gamma_2$.

 When $v > v_0$, the general solution of \eqref{eq:BCtw_linear} is $ \Phi^{\text{lin}}_{v > v_0}(z) := 1-A\mathrm{e}^{-\gamma_1z}-B\mathrm{e}^{-\gamma_2z}$ with exponential decay dominated by $\min \{ \gamma_1, \gamma_2 \}$. In the limit $v \to v_0$, the two roots $\gamma_1$ and $\gamma_2$ coincide and become equal to $\gamma_0$. This corresponds to the degenerate solution of \eqref{eq:BCtw_linear},
\begin{align}\label{eq:twsol_decay2}
    \Phi^{\text{lin}}_{v_0}(z) := 1-Az\mathrm{e}^{-\gamma_0z},
\end{align}
which asymptotically dominates any other exponentially decaying solution of the form \eqref{eq:twsol_decay1}. Substituting the minimal speed solution \eqref{eq:twsol_decay2} into \eqref{eq:BCtw_linear} gives
\begin{align*}
    (\lambda \mathrm{e}^{\gamma_0} + \omega - \lambda - v_0\gamma_0)z - \lambda \mathrm{e}^{\gamma_0} + v_0 = 0,
\end{align*}
then solving the coefficients gives the result, \eqref{eq:v0=lambdaExpgamma0} and \eqref{eq:gamma_0equation}.
\end{proof}

\begin{remark}
The solution $\gamma_0$, of the transcendental equation \eqref{eq:gamma_0equation}, is given in terms of the principle branch $W_0$ of the Lambert W function \cite{LambertWfunction_1996}. This has a unique positive solution
\begin{align*}
    \gamma_0 := 1 + \text{W}_0\left(\frac{\omega-\lambda}{\lambda \mathrm{e}}\right) > 0, \quad \lambda,\omega>0.
\end{align*}
\end{remark}

\subsection{Logarithmic correction of the front} We employ a formal matched asymptotic analysis to obtain the position and shape at time $t$ of a front evolving according to the dynamics defined by \eqref{eq:BC}. This analysis closely follows \cite{BrunetDerrida_1997}. In particular, we will recover a logarithmic shift in the position of the front.

\begin{thm}
\label{thm:tcorrection}
Define the position of the front of \eqref{eq:BC} as
\begin{align*}
    X_t := \sum_{n=0}^\infty (1-\varphi_n(t)), \quad t \geq 0.
\end{align*}
Suppose that $\lim_{t \to \infty} X_t/t = v_0$ with $v_0$ defined by \eqref{eq:v0=lambdaExpgamma0}. Then
\begin{align}\label{eq:Xt_logshift}
    X_t \sim v_0 t - \frac{3}{2\gamma_0} \log t, \quad \text{as $t \to \infty$}.
\end{align}
\end{thm} 

\begin{lm}
    The position of the front $X_t$ is positive and finite for all $t > 0$.
\end{lm}
\begin{proof}
    Positivity is clear from property (i) of \Cref{lm: varphi_n properties}. To prove that $X_t$ converges, we recall the inequality \eqref{ineq: 0<1-varphi<f} which after summation implies
    \begin{align*}
        X_t \leq \e^{\omega t}(\e^{\lambda t} - 1) < \infty \quad \text{for all } t \geq 0,
    \end{align*}
    which gives the result.
\end{proof}

\begin{remark}
    It is well known that for equations similar in form to \eqref{eq:BC}, the travelling wave speed is influenced by the tail behaviour of the initial data. In fact, for initial data which decays fast enough, the travelling wave speed is minimal among all possible travelling wave solutions, see \cite{AronsonWeinberger_1978, DeeLanger_1983, Bramson_1983, vanSaarloosWim_1987, vanSaarloosWim_1988, vanSaarloosWim_1989, ColletEckmann_1990}. In view of this, and the step-like initial condition of \eqref{eq:BC}, we assume convergence to the minimal speed travelling wave solution. In fact, this is what we observe numerically, see \Cref{fig:logtcorrection}.
\end{remark}

\begin{remark}
    There exists technical frameworks which prove rigorously the convergence to minimal speed and leading order corrections. The logarithmic correction \eqref{eq:Xt_logshift} was first identified and rigorously proved for the maximal displacement of branching Brownian motion (BBM) and then extended to the location of the level sets of solutions to the Fisher-KPP equation in the seminal works of Bramson in the 1980s \cite{Bramson_1978, Bramson_1983}. More recently, for the Fisher-KPP level sets, this was reproved by pure PDE methods \cite{HamelNolen_2013}. The same phenomenon has subsequently been rigorously proved for branching random walks (BRW) \cite{BRW_BramsonCorrection_2009, BRW_BramsonCorrection_2013}, where the maximum or minimum particle position exhibits an analogous logarithmic shift, and for discrete Fisher-KPP equations on the integer lattice $\Z$ \cite{discreteFKPP_BramsonCorrection_2023}. Due to the length and technical nature of these proofs, we will instead employ a formal asymptotic argument adapted from \cite{BrunetDerrida_1997}.
\end{remark}

\subsection{Proof of \Cref{thm:tcorrection}} Assuming convergence to minimum speed, we write
\begin{align}\label{eq:Xt=vt-ct}
    X_t = v_0 t - c(t), \quad t \geq 0,
\end{align}
where $c(t)$ is a correction term to be determined and $v_0$ is defined by \eqref{eq:v0=lambdaExpgamma0}.

\subsubsection{Scaling form ansatz} We assume the scaling ansatz
\begin{align}\label{eq:generalscaleG}
    \varphi_n(t) = 1 - t^\alpha G(z) \mathrm{e}^{-\gamma_0 (n-X_t)}, \quad z = (n-X_t)t^{-\alpha},
\end{align}
where $G$ is an amplitude to be determined. The idea behind this choice of scaling form is that we suppose there exists some constant $\alpha>0$ such that there is a region of size $t^\alpha$ ahead of the front $X_t$ that keeps the memory of the initial condition $\varphi_n(0)=1$. This choice \eqref{eq:generalscaleG} is analogous to that in \cite{BrunetDerrida_1997}, and we will see that it is asymptotically valid for large $t$.

\subsubsection{Solving for the amplitude} Inserting the scaling form ansatz \eqref{eq:generalscaleG} into \eqref{eq:BC}, and using the minimum speed $v_0$ \eqref{eq:v0=lambdaExpgamma0}; the definition of $\gamma_0$ \eqref{eq:gamma_0equation}; and the Taylor expansion
\begin{align*}
    G(z-t^{-\alpha}) &= G(z) - t^{-\alpha} G(z)' + \frac{1}{2}t^{-2\alpha}G''(z)  + \mathcal{O}(t^{-3\alpha}),
\end{align*}
we compute, to leading order in $t$, that the amplitude $G$ satisfies 
\begin{multline}\label{eq:Gode}
    v_0 t^{-\alpha} G''(z) + 2\alpha t^{\alpha-1} \left( zG'(z) - G(z) \right) + 2 t^\alpha \gamma_0 \dot{c}(t)G(z)\\  = 2\dot{c}(t)G'(z) + 2\omega G(z)^2 t^{2\alpha} \mathrm{e}^{-\gamma_0 z t^\alpha} + \mathcal{O}(t^{-2\alpha}).
\end{multline}
Enforcing the requirement that the leading order terms of \eqref{eq:Gode} be comparable, we see that
\begin{align*}
     \alpha = \frac{1}{2}, \quad \dot{c}(t) \sim \frac{\beta}{\gamma_0 t}, \quad \beta \in \R.
\end{align*} 
In this scale, all terms on the left-hand-side of \eqref{eq:Gode} are of order $\mathcal{O}(t^{-1/2})$ whereas the right-hand-side is of order $\mathcal{O}(t^{-1})$ while on the half-line $z>0$. Hence the right-hand-side of \eqref{eq:Gode} is asymptotically negligible and the amplitude $G$ solves
\begin{align}\label{eq:Gode2}
    G''(z) + \frac{z}{v_0}G'(z) + \frac{1}{v_0}(2 \beta-1)G(z) = 0,
\end{align}
and the position of the front \eqref{eq:Xt=vt-ct} becomes
\begin{align}\label{eq:Xt=vt-betalnt}
    X_t \sim v_0t -\frac{\beta}{\gamma_0}\log t.
\end{align}
\begin{remark}
    We note that \eqref{eq:Gode2} and \eqref{eq:Xt=vt-betalnt} correspond to equations (A10) and (A11) in \cite{BrunetDerrida_1997} with $v_0=2$ and $\gamma_0=1$, whereas here $v_0$ and $\gamma_0$ are given by \eqref{eq:v0=lambdaExpgamma0} and \eqref{eq:gamma_0equation}.
\end{remark}

\begin{prop}
    The general solution of the second-order ODE \eqref{eq:Gode2} is given by linear combinations of the functions
    \begin{align}\label{eq:G_gensol}
    \left \{
    \mathrm{e}^{-z^2 / 2v_0}   {}_{1}F_{1} \left( 1-\beta ; \frac{1}{2} ; \frac{z^2}{2v_0} \right),\ \mathrm{e}^{-z^2 / 2v_0}   \frac{z}{\sqrt{2 v_0}}{}_1F_1 \left( \frac{3}{2}-\beta ; \frac{3}{2} ; \frac{z^2}{2v_0} \right) \right \},
\end{align}
where ${}_1F_1$ is the confluent hypergeometric function of the first kind. 
\end{prop}
\begin{proof}
    Substituting the ansatz $G(z) = u(z) \mathrm{e}^{-z^2/2v_0}$ into \eqref{eq:Gode2}, we obtain
    \begin{align*}
        u''(z) - \frac{z}{v_0}u'(z) + \frac{2}{v_0}(\beta-1)u(z) = 0.
    \end{align*}
    Then taking the coordinate transformation $w = \frac{z^2}{2v_0}$ and $y(w) = u(z)$, we obtain
    \begin{align*}
        w y(w)'' + \left(\frac{1}{2} - w\right)y'(w) + (\beta-1)y(w) = 0,
    \end{align*}
    which matches the form of Kummer's differential equation whose general solution is spanned by the functions
    \begin{align*}
        {}_1F_1\left(1-\beta ; \frac{1}{2} ; w\right) \quad \text{and} \quad w^{1/2} {}_1F_1\left(\frac{3}{2}-\beta ; \frac{3}{2} ; w\right),\quad \beta \in \R,
    \end{align*}
    see \cite[Section 2]{MathewsEsrick_2022}, \cite[Section 13]{AbramowitzStegun_1970}. The result then follows.
\end{proof}

\subsubsection{Selection of amplitude} As $t \to \infty$ the front will approach its limiting form and so its shape should look like the minimal speed solution $1-z\mathrm{e}^{-\gamma_0 z}$ which is \eqref{eq:twsol_decay2}. Hence we choose the solution of \eqref{eq:Gode2} which is linear at $z=0$. This corresponds to the linearly independent solution with odd powers of $z$ given by the second component of \eqref{eq:G_gensol}.

In view of this, the relevant solution of \eqref{eq:Gode2} is
\begin{align}
    R_{\beta}(z) &:= \mathrm{e}^{-z^2 / 2v_0}   \frac{z}{\sqrt{2 v_0}}{}_1F_1 \left( \frac{3}{2}-\beta ; \frac{3}{2} ; \frac{z^2}{2v_0} \right) 
    = \mathrm{e}^{-z^2 / 2v_0} \sum_{s=0}^\infty \frac{(3/2-\beta)_s}{(3/2)_s} \frac{1}{s!} \left( \frac{z}{\sqrt{2v_0}} \right)^{2s+1}, \label{eq:Rsol_infsum}
\end{align}
which converges for $z,\beta \in \mathbb{C}$ \cite[Equation 2.2]{MathewsEsrick_2022}, and the Pochhammer symbol is given by
\begin{align*}
    (a)_0 = 1 \quad \text{and} \quad (a)_s = a(a+1)\ldots(a+s-1) = \frac{\Gamma(a+s)}{\Gamma(a)} \quad \text{for $s \in \N$}.
\end{align*}

\subsubsection{Solving the correction $c(t) = \frac{\beta}{\gamma_0} \log(t)$} For the scaling form ansatz \eqref{eq:generalscaleG} to be reachable for the initial condition $\varphi_n(0)=1$, it must satisfy the bound $\varphi_n(t)\leq 1$ over its entire domain. This requires that the amplitude function $R_{\beta}(z)$ be positive on the half-line $z>0$. To proceed, we make use of the following lemma:

\begin{lm}[$R_\beta$ asymptotics]
    In the limit $z \to \infty$, 
    \begin{align}\label{eq:G2asymp}
    R_{\beta}(z) \sim 
    \begin{cases}
        \displaystyle \frac{\Gamma(3/2)}{\Gamma(3/2-\beta)} (2v_0)^{\beta-1/2}  z^{1-2\beta}, & \text{if $\beta \notin 3/2 + \N_0$} \\[4mm]
        \displaystyle (2v_0)^{1-\beta} \mathrm{e}^{-z^2/2v_0} z^{2\beta - 2}, & \text{if $\beta \in 3/2 + \N_0$},
    \end{cases}
\end{align}
\end{lm}
\begin{proof}
This follows immediately from known asymptotics for ${}_1F_1$ in the limit $z \to \infty$, see \cite[Equations 4.1, 4.3]{MathewsEsrick_2022} and \cite[Subsection 13.5.1]{AbramowitzStegun_1970}.
\end{proof}

\begin{prop}[Valid $\beta$ values]
    The amplitude function $R_{\beta}(z)$ is positive on the half-line $z>0$ if and only if $\beta \leq 3/2$.
\end{prop}
\begin{proof}

Using the power series definition \eqref{eq:Rsol_infsum}, one can show that $R_{\beta}(z)$ satisfies the recurrence relation
\begin{align}\label{eq:G''_recurrence}
    R_{\beta}''(z) = -\frac{2}{v_0}\frac{\Gamma(\beta+1)}{\Gamma(\beta)}R_{\beta+1}(z).
\end{align}

\noindent\textbf{Case 1: $\beta > 3/2$.} Consider the sub-cases:

\begin{itemize}
    \item  \textbf{Case 1a: $\beta \notin 3/2 + \N_0$.}\\
    The asymptotic form \eqref{eq:G2asymp} implies that $R_{\beta}$ is negative for large $z$ if $\beta \in (3/2,5/2) + 2\N_0$, due to the $\Gamma(3/2-\beta)$ term. This means that $R_{\beta}(z)$ changes sign from positive to negative at some point $z_*>0$, stays negative and then converges to zero. Then there must exist some $ r \in (z_*,\infty)$ such that $R_{\beta}''(r) > 0$, and relation \eqref{eq:G''_recurrence} implies that $R_{\beta+1}(r)<0$. Hence for all $\beta \in (3/2,5/2) + \N_0$, there exists some $r>0$ such that $R_{\beta}(r)<0$.\\

    \item \textbf{Case 1b: $\beta \in 3/2 + \N$.}\\
    For these values of $\beta$, the first argument of ${}_1F_1$ in \eqref{eq:Rsol_infsum} takes values in $\Z_{\leq -1}$ and is thus a finite polynomial in $z$. The power series definition \eqref{eq:Rsol_infsum} then becomes
\begin{align}
    R_{\beta \in 3/2 + \N}(z) = \mathrm{e}^{-z^2 / 2v_0} \sum_{s=0}^{|3/2-\beta|} \frac{(3/2-\beta)_s}{(3/2)_s}\frac{1}{s!} \left( \frac{z}{\sqrt{2v_0}} \right)^{2s+1}.\label{eq:G2sol_gamma0betaHalfint}
\end{align}

We show that \eqref{eq:G2sol_gamma0betaHalfint} is negative at some point for $z > 0$. Let $\beta = 3/2 + m$ for some $m \in \N$. If $m \in \{1,3,5,\ldots \}$ then the coefficient of the highest power term in \eqref{eq:G2sol_gamma0betaHalfint} is negative since $(3/2-\beta)_{|3/2-\beta|} = (-m)_m < 0$. Therefore \eqref{eq:G2sol_gamma0betaHalfint} becomes and stays negative for large $z$ and converges to zero. By the same argument as in Case 1a, the relation \eqref{eq:G''_recurrence} implies that $R_{\beta+1}$ becomes negative at some point which corresponds to values $m \in \{ 2,4,6,\ldots \}$. Hence for all $\beta \in 3/2 + \N$, there exists some $r>0$ such that $R_{\beta}(r)<0$.
\end{itemize}

\noindent\textbf{Case 2: $\beta \leq 3/2$.}\\
Using the power series definition \eqref{eq:Rsol_infsum}, the sign of the $s^{th} \geq 0$ coefficient is given by the sign of
    \begin{align*}
        (3/2 - \beta)_s = 
        \begin{cases}
            1, & \text{ if $s = 0$},\\
            \prod_{k=0}^{s-1} (3/2 - \beta + k) \geq 0, & \text{ if $s \geq 1$},
        \end{cases}
    \end{align*}
    where the case $s \geq 1$ evaluates to zero only if $\beta = 3/2$. Therefore $R_{\beta}(z)>0$ for all $z > 0$.

\end{proof}

 It follows that the scaling form ansatz \eqref{eq:generalscaleG} is valid only for $\beta \leq 3/2$. Consider the asymptotic form \eqref{eq:G2asymp} and suppose initially $\varphi_x(0) = 1 - x^v \e^{-\gamma_0x}$. If $\beta<3/2$ then it follows that $v = 1-2\beta$ and $v>-2$. Then for initial conditions with $v \leq -2$ or steeper, such as in our case of step initial conditions $\varphi_n(0)=1$, it follows that $\beta = 3/2$. Therefore, the shift in the position of the front \eqref{eq:Xt=vt-betalnt} with $\beta=3/2$ gives the result \eqref{eq:Xt_logshift} of \Cref{thm:tcorrection}. Furthermore, the leading order propagating speed towards the travelling wave solution of minimal speed is
\begin{align}\label{eq:speed_pred}
    v(t) \sim v_0 - \frac{3}{2\gamma_0}\log(t)/t.
\end{align}

\begin{figure}[H]
    \centering
    \includegraphics[width=0.6\textwidth]{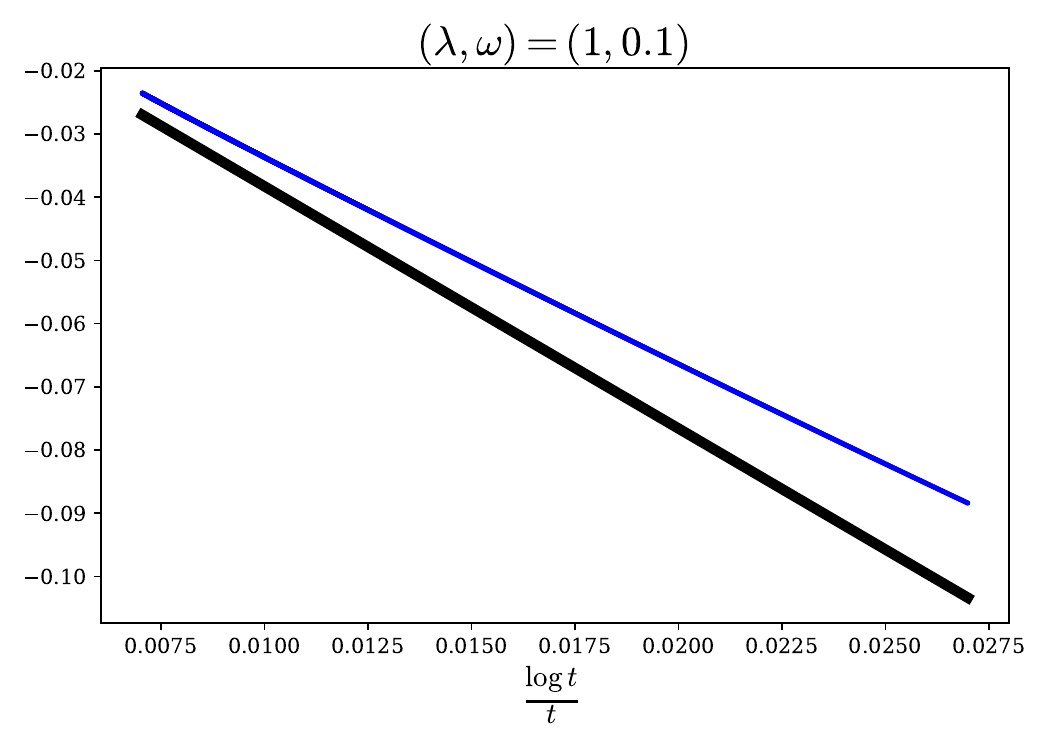}

    \caption{Numerical computations demonstrating the logarithmic correction \eqref{eq:speed_pred}. In blue: the numerical computation of $v(t)-v_0$ with $v_0$ defined by \eqref{eq:v0=lambdaExpgamma0}. In black: the line $-\frac{3}{2\gamma_0} \log t / t$ with $\gamma_0$ defined by \eqref{eq:gamma_0equation}.}
    \label{fig:logtcorrection}
\end{figure}

\subsection{Shape of limiting profile} The uniform convergence to a limiting profile of the solutions to the system of ODEs \eqref{eq:BC} was established in \Cref{sec: Convergence to limiting profile}. This was the result of \Cref{thm: u_n convergence}. We now identify an explicit analytic expression that serves as an excellent approximation of this limiting profile.

\subsubsection{Continuum approximation}
To proceed we assume the existence of a smooth function $f \in \Cf^\infty(\R \times [0,\infty))$ that serves as an analytic continuation of the sequence $(\varphi_n)_{n \geq 0}$, satisfying the interpolation embedding $f(nh,t) = \varphi_n(t)$ for all $(n,t) \in \N_0 \times \R_{\geq 0}$, where $h>0$ is the lattice embedding spacing. To ensure a well-defined continuum limit we also scale the parameters $\lambda,\omega$ with respect to the lattice spacing, i.e.  $\lambda \mapsto \lambda_h$ and $\omega \mapsto \omega_h$. 

Using a local Taylor series expansion about the continuous coordinate $x=nh$, with $h\ll 1$, we compute
\begin{align*}
    \varphi_n(t) - \varphi_{n-1}(t) = h \partial_x f(x,t) - \frac{h^2}{2} \partial_x^2 f(x,t) + \order(h^3),
\end{align*}
then inserting this expansion into \eqref{eq:BC} we obtain
\begin{align*}
    \partial_t f(x,t) = -\lambda_h \left(h \partial_x f(x,t) - \frac{h^2}{2} \partial_x^2 f(x,t) \right) - \omega_h f(x,t)(1-f(x,t)) + \order(\lambda_h h^3).
\end{align*}
To eliminate the advection term we introduce a moving-frame coordinate with reflection   $\bar{f}(x,t) := 1 - f(x+\lambda_h h t,t)$, which then gives us
\begin{align*}
    \partial_t \bar{f}(x,t) = \lambda_h \frac{h^2}{2} \partial_x^2 \bar{f}(x,t) + \omega_h \bar{f}(x,t)(1 - \bar{f}(x,t)) + \order(\lambda_h h^3).
\end{align*} 
If we choose the parameter scaling $\lambda_h := \frac{\lambda}{h^2}$ and $ \omega_h := \omega$, then pass to the continuum limit $h \to 0$, the right-hand-side becomes of order $\order(\lambda h)$ and we obtain
\begin{align}\label{eq: continuum limit FKPP}
    \partial_t \bar{f}(x,t) =  \frac{\lambda}{2} \partial_{x}^2 \bar{f}(x,t) + \omega \bar{f}(x,t)(1 - \bar{f}(x,t)),
\end{align}
which is the classical Fisher-KPP equation. 
Standard theory dictates that a step-like initial condition of \eqref{eq: continuum limit FKPP} will evolve into a unique travelling wave with minimal speed $\sqrt{2\lambda\omega}$, see \cite{AronsonWeinberger_1978, DeeLanger_1983, Bramson_1983, vanSaarloosWim_1987, vanSaarloosWim_1988, vanSaarloosWim_1989, ColletEckmann_1990}. 

We transform \eqref{eq: continuum limit FKPP} to travelling wave coordinates (with an additional $\lambda$-drift and a reflection). The additional $\lambda$-drift and reflection is included to account for the moving-frame and reflection that was used to obtain \eqref{eq: continuum limit FKPP}.

\begin{prop}
    Denote $g(z) := 1 - \bar{f}(x,t)$ where $z := x-(v_g-\lambda)t$. Then the continuum ODE \eqref{eq: continuum limit FKPP} becomes the boundary value problem
    \begin{equation}
        \left\{
            \begin{aligned}\label{eq: BC_gode_approx}
    &(v_g-\lambda)g'(z) = - \frac{\lambda}{2}g''(z) + \omega g(z)(1-g(z)), \quad z \in \R,\\
             &\lim_{z \to -\infty} g(z) = 0, \quad \lim_{z \to \infty} g(z) = 1.
    \end{aligned}
        \right.
    \end{equation}
    For the travelling wave speed
\begin{align}\label{eq: gsolexact_wavespeed}
    v_g = \lambda + \frac{5}{6}\sqrt{3\lambda\omega},
\end{align}
the ODE \eqref{eq: BC_gode_approx} admits a family of explicit translation-invariant solutions
\begin{equation}\label{eq: gsolexact}
    g(z;z_0) = 1 - \left(1+\exp \left(\sqrt{\tfrac{\omega}{3\lambda}}(z-z_0)\right)\right)^{-2}, \quad z \in \R.
\end{equation}
\end{prop}
\begin{proof}
    This follows by a computation due to Ablowitz and Zeppetella (1997) \cite{AblowitzZeppetella_1979}.
\end{proof}

\begin{figure}[H] 
\centering\includegraphics[width=0.6\linewidth]{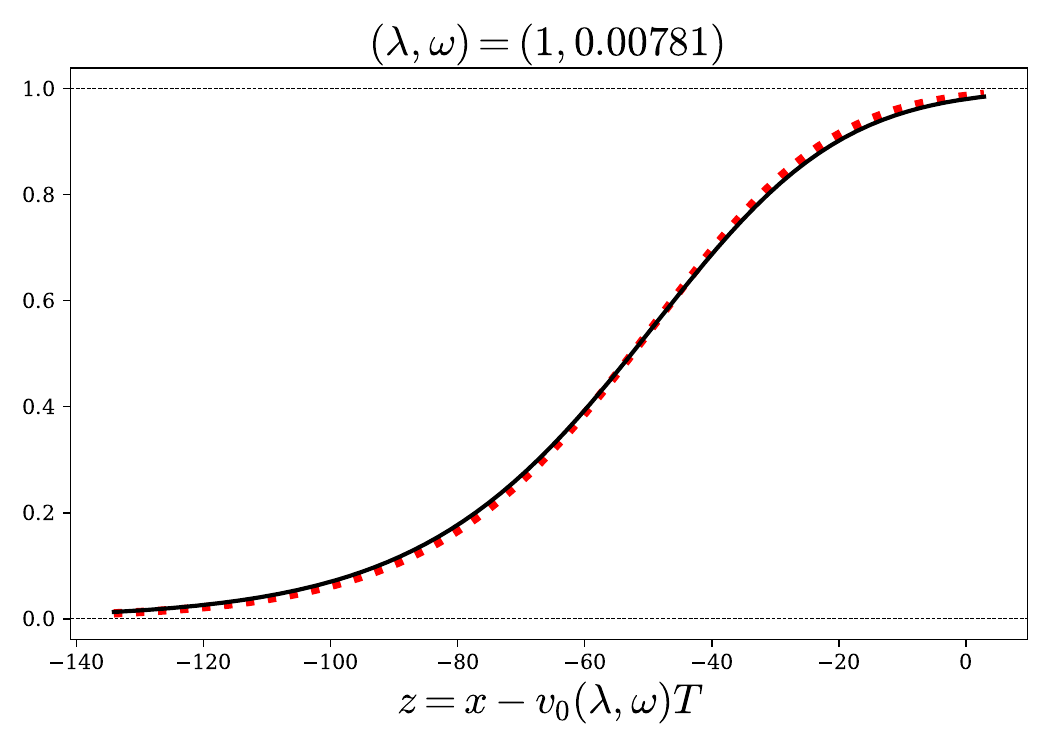}
  \caption{In black: the Ablowitz-Zeppetella exact solution \eqref{eq: gsolexact}. In red dotted: the numerical solution $\varphi_{z+v_0(\lambda,\omega)T}(T)$ of \eqref{eq:BC} for $T=2000$, where $v_0$ is given by \eqref{eq:v0=lambdaExpgamma0}. We have chosen $(\lambda,\omega) = (1, 0.00781)$ so that the wave speeds \eqref{eq:v0=lambdaExpgamma0} and \eqref{eq: gsolexact_wavespeed} equate,  $v_0(\lambda,\omega) = v_g(\lambda,\omega)$, as defined by \eqref{eq: equating exact and min BC speed}.}
  \label{fig: exact comparison}
\end{figure}

\subsubsection{Comparison of the exact solution \eqref{eq: gsolexact} with the large-time solution of \eqref{eq:BC}}
Analogously to the earlier analysis in \Cref{ss: min_speed_selection}, it is not hard to show that the minimum speed for a heteroclinic orbit of the boundary value problem \eqref{eq: BC_gode_approx} is given by $\lambda + \sqrt{2\lambda\omega}$. The Painlev\'e solution of \eqref{eq: BC_gode_approx} that is relevant to \eqref{eq:BC}, that is which satisfies the boundary conditions $g(-\infty)=0$ and $g(\infty) = 1$, is given by \eqref{eq: gsolexact} and has a wave speed of $v_g = \lambda + \frac{5}{6}\sqrt{3\lambda\omega}$ given by \eqref{eq: gsolexact_wavespeed}. Hence this solution is \textbf{not} a minimum speed solution. Indeed, for large $z$, the solution \eqref{eq: gsolexact} is of the form \eqref{eq:twsol_decay1} and not \eqref{eq:twsol_decay2}.

However, the minimum speed $v_0$, defined by \eqref{eq:v0=lambdaExpgamma0}, of \eqref{eq:BC} is larger than $\lambda+\sqrt{2\lambda\omega}$, and there exists a value of $\omega/\lambda$ where the speed of the Painlev\'e solution \eqref{eq: gsolexact} is equal to the minimum speed $v_0$ of \eqref{eq:BC}. This value is determined by equating the two speeds, \eqref{eq:v0=lambdaExpgamma0} and \eqref{eq: gsolexact_wavespeed}, i.e.
\begin{equation}\label{eq: equating exact and min BC speed}
 \lambda \exp\left(1+ W_0\left(\frac{\omega-\lambda}{\mathrm{e}\lambda}\right) \right)= \lambda + \frac{5}{6}\sqrt{3\lambda\omega}\quad \Rightarrow \quad \omega/\lambda \approx 0.00781\ldots
\end{equation}
\Cref{fig: exact comparison} shows excellent agreement between \eqref{eq: gsolexact} and the corresponding large-time numerical solution of \eqref{eq:BC}.

\section{Shift in speed for finite number of miners}\label{sec: shift in speed for finite miners}
 In this section we obtain, by way of formal matched asymptotics, the leading order correction of the difference $v_0-v_N$, where $v_0$ is defined by \eqref{eq:v0=lambdaExpgamma0} and $v_N$ is the travelling wave speed of the finite miner distribution function $\varphi^N$, with explicit definitions given by
 \begin{align}\label{eq: finite N-miner speed and distribution function}
     v_N := \lim_{t \to \infty} \frac{1}{t}\sum_{n=0}^\infty \left(1 - \varphi_n^N(t)\right), \quad \text{and} \quad \varphi_n^N(t) := \sum_{i=0}^{n} \theta_i^N(t), \quad t \geq 0, \: n \in \N_0,
 \end{align}
 in terms of the normalised finite $N$-miner $\theta^N$-process defined by \eqref{dd1} and \eqref{dd2} in \Cref{sec: height dynamics}.

\begin{thm}\label{thm:Ncorrection}
The leading order correction in the travelling wave speed $v_N$ for finite number of miners is given by
\begin{align}\label{eq: finite miner logNcorrection}
    v_0-v_N \simeq \frac{\pi^2 \gamma_0^2}{2} v''(\gamma_0) \frac{1}{( \log N )^2} \quad \text{as } N \to \infty,
\end{align}
where $v(\gamma)$ is defined in \eqref{eq:v(gamma)}, $v_0$ and $\gamma_0$ in \eqref{eq:v0=lambdaExpgamma0} and \eqref{eq:gamma_0equation}.
\end{thm}

The remainder of this section is devoted to a formal matched asymptotics proof of \Cref{thm:Ncorrection}. Our analysis closely follows \cite{BrunetDerrida_1997} and the final result is supported with numerical simulations of $v_N$ by simulating the $\theta^N$-process defined in \Cref{sec: height dynamics}.

\begin{remark}
    The right-hand-side of \eqref{eq: finite miner logNcorrection} is known in the literature as Brunet–Derrida behaviour. It describes the slow  $\order((\log N)^{-2})$-convergence of the propagation speed of stochastic, finite $N$-population travelling fronts toward the deterministic speed of its limiting $N=\infty$ counterpart. Its origins date back to Brunet-Derrida \cite{BrunetDerrida_1997}, and their later works \cite{BrunetDerrida_1999, BrunetDerrida_2001}, whereby they used formal matched asymptotics to show that imposing a small cutoff $\varepsilon$, on the leading edge of a Fisher-KPP travelling wave, produces a shift in speed proportional to $(\log \varepsilon)^{-2}$. This result was later proven rigorously for a broad class of cut-off functions \cite{DumortierPopovicKaper_2007}; and rigorous upper/lower bounds on the correction were obtained in \cite{BenguriaDepassier_2007, BenguriaDepassierLoss_2008}; and in \cite{MuellerMytnikQuastel_2011} they proved rigorously  the correction for a class of reaction-diffusion equations of KPP type perturbed by Fisher-Wright white noise. Brunet-Derrida \cite{BrunetDerrida_1997} also argued that for a particle population of size $N$, the correction scales as $(\log N)^{-2}$. Their physics interpretation is that the microscopic geometry of a finite population effectively cuts off the extremely low-density $\varepsilon = N^{-1}$ leading edge of the deterministic wave, and therefore reducing its speed. This connection became particularly influential in \cite{Berard_2008, BerardGourere_2010}, whereby they proved rigorously the $\order((\log N)^{-2})$-convergence for branching-selection particle systems on $\Z$ and $\R$. 
\end{remark}

\subsection{Proof of \Cref{thm:Ncorrection}}
To obtain the leading order correction \eqref{eq: finite miner logNcorrection}, we show
\begin{align}\label{eq: finite miner logeps correction}
    v_0-v_\varepsilon \simeq \frac{\pi^2 \gamma_0^2}{2} v''(\gamma_0) \frac{1}{(\log \varepsilon)^2} \quad \text{as } \varepsilon \to 0,
\end{align}
with $v_0$ and $\gamma_0$ defined in \eqref{eq:v0=lambdaExpgamma0} and \eqref{eq:gamma_0equation}, and $v_\varepsilon$ defined as the travelling wave speed of the $\varepsilon$-cutoff modified system of ODEs given by \eqref{eq:BC_eps_cutoff} below. The correction \eqref{eq: finite miner logNcorrection} then follows by an argument that $\varepsilon=N^{-1}$ which we now make precise.

\subsubsection{Cutoff due to discreteness $(\varepsilon=N^{-1})$}
Recall the finite miner distribution function $\varphi^N$, defined by \eqref{eq: finite N-miner speed and distribution function}, which has initial condition $\varphi_n^N(0) = 1$ for all $n \in \N_0$. Then at any time $t \geq 0$, there exists a unique height $K_t \in \N_0$ such that
\begin{align*}
    \varphi^N_k(t)  
    \begin{cases}
        = 1, \quad \text{if } k \geq K_t\ ,\\
        < 1, \quad \text{if } k < K_t\ ,
    \end{cases}
\end{align*}
which follows from the finite miner transitions \eqref{dd1} and \eqref{dd2} defined in \Cref{sec: height dynamics}. Hence we define $K_t := \min \{ k \in \N_0 \; | \; \varphi^N_k(t) = 1 \}$ which corresponds to the maximum height of all $N$ miners at time $t$. Due to the initial condition, $K_t$ is monotone non-decreasing.

At any time $t \geq 0$, at the leading edge of the front there is a group of miners of equal and maximum height $K_t$. Then there exists some unique later time $t+\delta$ such that one miner at this maximum height $K_t$ will mine a new block and be at height $K_{t+\delta} = K_t + 1$, i.e.
\begin{align*}
\varphi^N_{k}(t+\delta) = 
    \begin{cases}
        1 - \tfrac{1}{N}, \quad &\text{if $k = K_t$},\\
        1, \quad &\text{if $k \geq K_t + 1$}.
    \end{cases}
\end{align*}
Hence, in the language of \cite{BrunetDerrida_1997}, it makes sense that introducing a cutoff level of $\varphi^N_n = 1-\varepsilon$ corresponds exactly to $\varepsilon = N^{-1}$. Therefore, we define the $\varepsilon$-cutoff modified system of ODEs
\begin{align}\label{eq:BC_eps_cutoff}
    \dot{\varphi}_n = -\lambda(\varphi_n - \varphi_{n-1}) - \omega_\varepsilon \varphi_n (1 - \varphi_n) \Theta(1-\varepsilon-\varphi_n),
\end{align}
where $\omega_\varepsilon:=\omega(1+\varepsilon)$ due to \eqref{eq: avg drift ODEs}, and the cut-off function is defined by $\Theta(x):= \mathds{1}_{x > 0}$. We will use this cutoff system of ODEs \eqref{eq:BC_eps_cutoff} to derive \eqref{eq: finite miner logeps correction} which then, due to the correspondence $\varepsilon=N^{-1}$, will give the final result \eqref{eq: finite miner logNcorrection}.

\subsubsection{Approximate solutions in regions}
In what follows, we recall $v_\varepsilon$ as the travelling wave speed for the $\varepsilon$-cutoff system of ODEs \eqref{eq:BC_eps_cutoff}. A simple argument in \cite{BrunetDerrida_1997} gives 
\begin{align}\label{ineq: lambda<v_eps<=v_0}
    0 < \lambda < v_\varepsilon \leq v_0, \quad \text{for all } 0 \leq \varepsilon \ll 1.
\end{align}
If we convert the $\varepsilon$-cutoff system of ODEs \eqref{eq:BC_eps_cutoff} to travelling wave coordinates, the travelling wave shape $\Phi_{v_\varepsilon}(z) := \varphi_{z+v_\varepsilon t}(t)$ satisfies the boundary value problem
\begin{equation}
\label{eq:truncatedDDE}
    \left\{ 
    \begin{aligned}
        &v_\varepsilon \Phi_{v_\varepsilon}'(z) = \lambda(\Phi_{v_\varepsilon}(z)-\Phi_{v_\varepsilon}(z-1) + \omega_\varepsilon \Phi_{v_\varepsilon}(z)(1-\Phi_{v_\varepsilon}(z))\Theta(1-\varepsilon-\Phi_{v_\varepsilon}), \quad z \in \R,\\
        &\Phi_{v_\varepsilon}(-\infty)=0, \quad \Phi_{v_\varepsilon}(\infty) = 1.
    \end{aligned}
    \right.
\end{equation}
Following \cite{BrunetDerrida_1997}, we now consider the travelling wave shape $\Phi_{v_\varepsilon}$ in three regions:

\setlength{\itemsep}{6pt}
\begin{itemize}
    \item []\textbf{Region I}  where $0\le \Phi_{v_\varepsilon}(z) < 1-\varepsilon'$.
    \item []\textbf{Region II} where $1-\varepsilon' < \Phi_{v_\varepsilon}(z) < 1-\varepsilon$,\quad $\varepsilon<\varepsilon'$.
    \item []\textbf{Region III} where $1-\varepsilon < \Phi_{v_\varepsilon}(z) < 1$.
\end{itemize}

\begin{lm}[ Region I]
In Region I, the shape of the front $\Phi_{v_\varepsilon}$ coincides with $\Phi_{v_0}$, as defined in \eqref{eq:twsol_decay2}, so we have
\begin{align*}
    \Phi_{v_\varepsilon}(z) \simeq \Phi_{v_0}(z) \simeq 1 - Az\mathrm{e}^{-\gamma_0 z} \quad \text{as } z \to \infty.
\end{align*}
\end{lm}

In Regions II and III, the front $\Phi_{v_\varepsilon}$ is close to $1$ and its leading order asymptotic behaviour for large $z$ is given by the solutions to the linearised problem \eqref{eq:truncatedDDE}, which become
\begin{align}
    v_\varepsilon \Phi_{v_\varepsilon}'(z) &= \lambda(\Phi_{v_\varepsilon}(z)-\Phi_{v_\varepsilon}(z-1)) + \omega(1+\varepsilon) (1-\Phi_{v_\varepsilon}(z)), &&\text{in Region II}, \label{eq: s_eps ode region2}\\
    v_\varepsilon \Phi_{v_\varepsilon}'(z) &= \lambda(\Phi_{v_\varepsilon}(z)-\Phi_{v_\varepsilon}(z-1)), &&\text{in Region III}, \label{eq: s_eps ode region3}
\end{align}
and are connected by the intermediate boundary condition $\Phi_{v_\varepsilon} (z_\varepsilon) := 1-\varepsilon$. In what follows we define the shift in speed as 
\begin{equation*}
\Delta:=v_0-v_\varepsilon \geq 0,
\end{equation*}
with nonnegativity due to \eqref{ineq: lambda<v_eps<=v_0}, and assume $\Delta \gg \varepsilon$ which we will justify {\it a posteriori} from our final result \eqref{eq: epsilon correction}.

\begin{lm}[Region II]
For sufficiently small $\Delta$, the solution in Region II is given to leading order by
\begin{align}\label{eq: R2 sol}
\Phi_{v_\varepsilon} (z) \simeq
         1 - C \mathrm{e}^{-\alpha z}  \sin (\beta z + D)  \quad &\text{as } z \to \infty,
\end{align}
where $C$ and $D$ are constants and $\alpha =\gamma_0 + \Theta{(\Delta)}$ and $\beta = \Theta{(\Delta^{1/2})}$.
\end{lm}

\begin{proof}
Consider an exponential solution $\Phi_{v_\varepsilon}(z) = 1 - B\mathrm{e}^{-\gamma z}$ for Re$(\rho)>0$, which for \eqref{eq: s_eps ode region2} yields the associated characteristic equation
\begin{align}\label{eq: chareq R2}
    \lambda \mathrm{e}^{\gamma} - (v_0-\Delta) \gamma = \lambda - \omega(1+\varepsilon).
\end{align}

In the limit $\Delta \to 0$ we will have $\gamma=\gamma(\Delta) \to \gamma_0$ so we may define, without loss of generality, that $\gamma = \gamma_0 + \delta$ such that $\lim_{\Delta \to 0} \delta(\Delta) = 0$. Expanding the left-hand-side of \eqref{eq: chareq R2} as a series in $\delta$ and $\Delta$ we obtain up to leading order that
\begin{align}\label{eq: chareq R2 delta expand}
    \gamma_0 \Delta + \tfrac{1}{2}v_0 \delta^2 + \delta \Delta +\omega\varepsilon + \mathcal{O}(\delta^3) = 0,
\end{align}
where we made use of the definitions of $v_0$ and $\gamma_0$ in \eqref{eq:v0=lambdaExpgamma0} and \eqref{eq:gamma_0equation}. Writing $\delta= a \pm i b$, for $a \in \R$ and $b \in \R_{\geq 0}$, we obtain from \eqref{eq: chareq R2 delta expand} that
\begin{equation*}
\begin{split}
    \tfrac{1}{2}v_0(a^2-b^2) &= -(a+\gamma_0)\Delta -\omega\varepsilon + \order(|\delta|^3) \\
    a &= -\tfrac{\Delta}{v_0} + \order(|\delta|^3).
\end{split}
\end{equation*}
Truncating the order $\delta^3$ terms and using $\Delta \gg \varepsilon$, we solve for $b \in \R_{\geq 0}$ and obtain
 \begin{align*}
    b = \sqrt{\tfrac{2}{v_0}\gamma_0\Delta - \tfrac{\Delta^2}{v_0^2}} \sim \sqrt{\tfrac{2\gamma_0}{v_0}}\Delta^{1/2} + \order(\Delta^{3/2}),
\end{align*}
and therefore, in the limit $\Delta \to 0$, we obtain to leading order that
\begin{align}\label{eq: rho(Delta) leading order}
    \gamma \sim \gamma_0 -\tfrac{\Delta}{v_0} \pm i \sqrt{\tfrac{2\gamma_0}{v_0}} \Delta^{1/2}.
\end{align}
Writing $\gamma = \alpha \pm i \beta$ for $\alpha \in \R$ and $\beta \in \R_{\geq 0}$, we assume from here on, in view of \eqref{eq: rho(Delta) leading order}, that $\Delta$ is sufficiently small such that $\alpha>0$. Hence, from \eqref{eq: rho(Delta) leading order}, we obtain
\begin{align}
\label{eq:alphabetadef}
    \gamma_0 - \alpha \sim \frac{\Delta}{v_0} \: \text{ and } \: \beta \sim \sqrt{\tfrac{2\gamma_0}{v_0}} \Delta^{1/2} \quad \text{as } \Delta \to 0,
\end{align}
so that the solution of \eqref{eq: s_eps ode region2} up to leading order in $\Delta^{1/2}$ is given as stated.
\end{proof}

\begin{lm}[Region III]
For sufficiently small $\Delta$ the solution in Region III is given in leading order by
\begin{align}\label{eq: R3 sol}
\Phi_{v_\varepsilon} (z) \simeq
         1 - \varepsilon\, \mathrm{e}^{-\rho(z-z_\varepsilon)} \quad &\text{as } z \to \infty,
\end{align}
with
\begin{equation*}
\varrho \sim \varrho_0  - \frac{\varrho_0 \Delta}{\lambda +v_0(\varrho_0-1)},
\end{equation*}
and $\varrho_0$ is the real positive root of
\begin{align}\label{eq: chareq varrho_0}
    \lambda \mathrm{e}^{\varrho_0} - v_0 \varrho_0 = \lambda.
\end{align}
\end{lm}

\begin{proof}
Consider again an exponential solution $\Phi_{v_\varepsilon}(z) = 1 - \varepsilon\,\mathrm{e}^{-\varrho z}$ for Re$(\varrho)>0$, which for \eqref{eq: s_eps ode region3} yields the associated characteristic equation
\begin{align}\label{eq: chareq R3}
    \lambda \mathrm{e}^{\varrho} - (v_0-\Delta) \varrho = \lambda.
\end{align}
In the limit $\Delta \to 0$ we will have $\varrho\to \varrho_0$ where $\varrho_0$ is the real positive root of \eqref{eq: chareq varrho_0}, which exists since $0 < \lambda < v_0$ due to the inequality \eqref{ineq: lambda<v_eps<=v_0}. We thus write $\varrho = \varrho_0 + \delta$ such that $\lim_{\Delta \to 0} \delta(\Delta) = 0$, and expand the left-hand-side of \eqref{eq: chareq R3} as a series in $\delta$ and $\Delta$. Using \eqref{eq: chareq varrho_0} we obtain
\begin{align}\label{eq: chareq R3 delta expand}
    (\lambda\e^{\varrho_0}-v_0)\delta + \tfrac{1}{2}\lambda\e^{\varrho_0} \delta^2 + \varrho_0\Delta + \Delta \delta + \order(\delta^3) = 0.
\end{align}
Writing $\delta = a \pm ib$ for $a \in \R$ and $b \in \R_{\geq 0}$ such that $a(\Delta), b(\Delta) \to 0$ as $\Delta \to 0$ and ignoring terms of $\order(\delta^3)$, we obtain for the imaginary part of \eqref{eq: chareq R3 delta expand} that
\begin{align*}
    ib ((a+1)\lambda\e^{\varrho_0}-v_0+\Delta)  = 0.
\end{align*}
Because $a\to 0$ as $\Delta\to 0$ and $v_0=\lambda \e^{\gamma_0}$ with $\gamma_0 \neq \varrho_0$ for $\omega>0$, we must have $b=0$.

For the real part of \eqref{eq: chareq R3 delta expand} we find that
\begin{equation}
\label{eq:aIII}
a \sim - \frac{\varrho_0 \Delta}{\lambda +v_0(\varrho_0-1)} + \order{(\Delta^2)},
\end{equation}
so that the solution of \eqref{eq: s_eps ode region3} up to leading order in $\Delta$ is given as stated.
\end{proof}

\subsubsection{Matching conditions between regions} 
There is a transition point between Region I and Region II where $\Phi_{v_\varepsilon}$ is continuous. Since in Region I we have $\Phi_{v_\varepsilon}(z) \simeq \Phi_{v_0}(z) \simeq 1 - A z \mathrm{e}^{-\gamma_0 z}$ we obtain, to leading order in $\Delta^{1/2}$ that $D=0$ and $C = A/\beta$ when matching with the solution \eqref{eq: R2 sol} in Region II.

Between region II and II we impose continuity of $\Phi_{v_\varepsilon}$ and its first derivative at the point $z = z_\varepsilon$, and obtain from \eqref{eq: R2 sol} and \eqref{eq: R3 sol} that
\begin{equation}
\begin{split}
    A \mathrm{e}^{-\alpha z_\varepsilon} \sin(\beta z_\varepsilon) &= \varepsilon \beta, \\
    A \mathrm{e}^{-\alpha z_\varepsilon} \left( -\alpha\sin(\beta z_\varepsilon) + \beta \cos(\beta z_\varepsilon) \right) &= -(\varrho_0+a) \varepsilon \beta,
\end{split}
\label{eq: R2&3 cond: continuity}
\end{equation}
with $\alpha=\gamma_0 + \order(\Delta)$ and $a$ given by \eqref{eq:aIII}, leading to the relation
\begin{align}\label{eq: eps_correction_tan_relation_1}
    \gamma_0 - \frac{\beta}{\tan(\beta z_\varepsilon)} - \varrho_0 = \order(\Delta).
\end{align}
Recall from \eqref{eq:alphabetadef} that $\beta = \order (\Delta^{1/2})$ and so up to leading order in $\Delta^{1/2}$ \eqref{eq: eps_correction_tan_relation_1} becomes
\begin{align}\label{eq: eps_correction_tan_relation_2}
    \tan(\beta z_\varepsilon) = \frac{-\beta}{\varrho_0 - \gamma_0} < 0,
\end{align}
since $\beta > 0$ and one can show that $\varrho_0 > \gamma_0$ from their definitions \eqref{eq:gamma_0equation} and \eqref{eq: chareq varrho_0}.

In order to satisfy \eqref{eq: eps_correction_tan_relation_2} we require 
\begin{align}\label{eq: tan_beta asymp}
    \beta z_\varepsilon \simeq \pi \quad \text{and} \quad \pi - \beta z_\varepsilon \simeq \tfrac{\beta}{\varrho_0 - \gamma_0} \quad \text{as } \beta \to 0.
\end{align}
Then recall \eqref{eq: R2&3 cond: continuity} and expanding around $\beta z_\varepsilon \simeq \pi$ and using \eqref{eq: tan_beta asymp} gives $A \mathrm{e}^{-\gamma_0 z_\varepsilon}\tfrac{\beta}{\varrho_0 - \gamma_0} \simeq \varepsilon \beta$ as $\beta \to 0$ which, after solving for $z_\varepsilon$, corresponds to
\begin{align*}
    z_\varepsilon \simeq - \tfrac{1}{\gamma_0} \log \varepsilon -\tfrac{1}{\gamma_0}  \left( \tfrac{\varrho_0-\gamma_0}{A} \right)  \quad \text{as } \beta \to 0.
\end{align*}
For small $\varepsilon$ the constant is irrelevant and the condition $\beta z_\varepsilon \simeq \pi$ in \eqref{eq: tan_beta asymp} gives
\begin{align*}
    \beta \simeq \frac{\pi \gamma_0}{\log \tfrac{1}{\varepsilon}} \quad \text{as } \beta \to 0.
\end{align*}
Then, as $\beta$ is small, the shift in speed $\Delta = v_0 - v_\varepsilon$ is obtained via series expansion, up to leading order, and given by
\begin{align}\label{eq: epsilon correction}
    v_0-v_\varepsilon \simeq \frac{1}{2} v''(\gamma_0) \beta^2 \simeq \frac{\pi^2 \gamma_0^2}{2} v''(\gamma_0) \frac{1}{(\log \varepsilon)^2} \quad \text{as } \varepsilon \to 0,
\end{align}
which is the result \eqref{eq: finite miner logeps correction} and combining with the identification $\varepsilon = N^{-1}$ concludes the proof of \Cref{thm:Ncorrection}. Note that the analysis above shows that $\Delta = \Theta(1/(\log \varepsilon)^2)$ for sufficiently small $\varepsilon$, which justifies the earlier assumption that $\Delta \gg \varepsilon$ in the derivation of \eqref{eq: rho(Delta) leading order}.


\begin{figure}[H]
    \centering
    \includegraphics[width=0.6\textwidth]{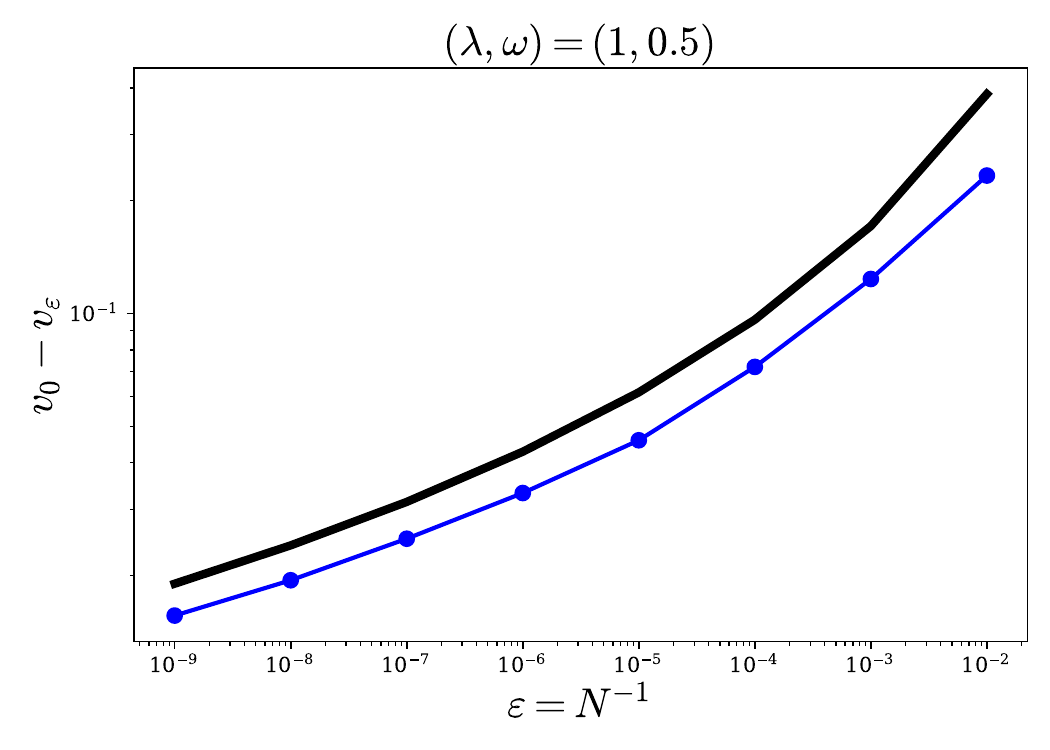}
    \caption{
    Numerical simulations demonstrating the logarithmic correction \eqref{eq: epsilon correction} with $\varepsilon = N^{-1}$. In blue dotted: the numerical simulation of $v_0 - v_{N^{-1}}$ with $v_0$ defined by \eqref{eq:v0=lambdaExpgamma0}. In black: the correction $\frac{\pi^2 \gamma_0^2}{2} v''(\gamma_0) \frac{1}{(\log \varepsilon)^2}$ with $\gamma_0$ defined by \eqref{eq:gamma_0equation}.
    }
    \label{fig: epsilon correction}
\end{figure}

\begin{appendix}

\section{Proofs}

\subsection{Proof of \Cref{prop:BC}}\label{proof: BC} The initial condition $\varphi_n(0)=1$ follows from \eqref{eq:BCcdf} and the initial condition $\theta_i(0) = \delta_{i,0}$ for $i \geq 0$. For $n=0$ we have $\varphi_0(t) = \theta_0(t)$ and so the result follows from \eqref{eq:theta1} and the boundary condition $\varphi_{-1} \equiv 0$. For $n \geq 1$, we use \eqref{eq:theta2}, \eqref{eq:BCcdf}, the normalisation $\sum_{i \geq 0} \theta_i(t) = 1, \: t \geq 0$, and compute
\begin{equation*}
\begin{split}
    \frac{d}{dt}\varphi_n(t) 
    &= -\lambda \left( \sum_{i=0}^n \theta_i(t) - \sum_{i=0}^n \theta_{i-1}(t) \right) - \omega \sum_{i=0}^n \theta_i(t) \left( \sum_{j=i+1}^\infty \theta_j(t) - \sum_{j=0}^{i-1} \theta_j(t) \right)\\
    &= -\lambda (\varphi_n(t) - \varphi_{n-1}(t)) - \omega \sum_{i=0}^n \theta_i(t) \left( 1 - \theta_i(t) - 2\sum_{j=0}^{i-1} \theta_j(t) \right)\\
    &= -\lambda (\varphi_n(t) - \varphi_{n-1}(t)) - \omega \varphi_n(t) + \omega \left( \sum_{i=0}^n \theta_i(t)^2 + 2\sum_{i=1}^{n} \theta_i(t) \left( \sum_{j=0}^{i-1} \theta_j(t)  \right) \right)\\
    &= -\lambda (\varphi_n(t) - \varphi_{n-1}(t)) - \omega \varphi_n(t) + \omega \left( \sum_{i=0}^n \theta_i(t) \right)^2\\
    &= -\lambda (\varphi_n(t) - \varphi_{n-1}(t)) - \omega \varphi_n(t)(1 - \varphi_n(t)), \quad t \geq 0.
\end{split}
\end{equation*}
The $n=0$ solution \eqref{eq:BCinitsol} is obtained by separation of variables using $\varphi_{-1} \equiv 0$.
\qed

\subsection{Proof of \Cref{lm: varphi_n properties}}\label{proof: BCsolproperties}

\subsubsection{Proof of $\mathrm{(i)}$} We proceed by induction on $n \geq 0$. The case $n=0$ is clear from the explicit solution $\varphi_0$. For the inductive step, suppose $\varphi_{n-1}(t) \in [0,1]$ for all $t\geq0$ for some $n \geq 1$. We show  $\varphi_{n}(t) \in [0,1]$ for all $t\geq0$. This can be done by looking at the derivative $\dot \varphi_n$ on the boundary of $[0,1]$. From \eqref{eq:BC}, it follows that
\begin{align}
    &\dot \varphi_n  = -\lambda (1-\varphi_{n-1}) \in [-\lambda, 0] \; \text{if $\varphi_n = 1$},\quad \text{and} \label{der1}\\
    &  \quad \dot \varphi_n  = \lambda \varphi_{n-1} \in [0,\lambda]  \; \text{if $\varphi_n = 0$},\label{der2}
\end{align}
since $\varphi_{n-1}(t) \in [0,1], \: t \geq 0$ by induction hypothesis. From (\ref{der1}), if $\varphi_n(t) = 1$ then $\varphi_n(t+\varepsilon) \leq 1$ for some small $\varepsilon$. This gives $\varphi_n(t) \leq 1$ for all $t\geq 0$. From (\ref{der2}), if $\varphi_n(t) = 0$ then $\varphi_n(t+\varepsilon) \geq 0$ for some small $\varepsilon$. This gives $\varphi_n(t) \geq 0$ for all $t\geq 0$. We conclude $\varphi_n(t) \in [0,1]$ for all $t\geq 0$. The result extends to all $n \geq 0$ by induction.

\subsubsection{Proof of $\mathrm{(ii)}$} Define $\delta_n(t) := \varphi_n(t) - \varphi_{n-1}(t)$. We prove  $\delta_n(t) \geq 0$ for all $t\geq0$ by induction on $n \geq 0$. The case $n=0$ is clear from the explicit solution $\varphi_0$ and $\varphi_{-1} \equiv 0$. For the inductive step, suppose $\delta_{n-1}(t) \geq 0$ for all $t\geq0$ for some $n \geq 1$. We show $\delta_{n}(t) \geq 0$ for all $t\geq0$. From \eqref{eq:BC}, it follows that
\begin{equation*}
    \dot{\delta}_n(t) = -(\lambda+\omega)\delta_n(t) + \lambda \delta_{n-1}(t) + \omega\delta_n(t)(\varphi_n(t) + \varphi_{n-1}(t)),
\end{equation*}
and therefore $\dot{\delta}_n = \lambda \delta_{n-1}$ if $\delta_n = 0$. We know $\delta_n(0) = 0$ for all $n\geq0$ by the initial condition $\varphi_n(0) = 1$. Suppose $\delta_n(t) = 0$, then $\dot{\delta}_n(t) = \lambda \delta_{n-1}(t) \geq 0$ since $\delta_{n-1}(t) \geq 0$ by induction hypothesis. Hence $\delta_n(t) \geq 0$ for all $t\geq 0$. The result extends to all $n \geq 0$ by induction.

\subsubsection{Proof of $\mathrm{(iii)}$} Follows from (i), (ii), and observing the right-hand-side of \eqref{eq:BC}.

\subsubsection{Proof of smoothness and uniqueness}  We proceed by induction on $n \geq 0$. The result holds for $n=0$ due to the explicit solution $\varphi_0$. For the induction step, assume the existence of a unique smooth solution $\varphi_{n-1}$ of \eqref{eq:BC} on $t \geq 0$ for some $n \geq 1$. Let $y := \varphi_n$ and denote the right-hand-side of \eqref{eq:BC} by
\begin{align*}
    F(t,y) = -\lambda(y - \varphi_{n-1}(t)) - \omega y (1 - y).
\end{align*}
By Property (i), we infer $-\lambda-\omega < F(t,y) \leq 0$ for all $(t,y) \in \R_{\geq 0} \times [0,1]$. Then $F$ is smooth and bounded in $y \in [0,1]$  so is globally Lipschitz continuous in $y$. Secondly, $F$ is smooth in $t \geq 0$ by the induction hypothesis. The result follows from the Picard-Lindel\"of theorem. 

\subsubsection{Proof of $\mathrm{(iv)}$} The initial solution \eqref{eq:BCinitsol} is asymptotically
    \begin{align*}
        \varphi_0(t) \sim \frac{\lambda+\omega}{\lambda} \mathrm{e}^{-(\lambda+\omega)t} \quad \text{as $t \to \infty$}.
    \end{align*}
    We proceed by induction on $n \geq 1$ and assume that $\varphi_{n-1}$ is asymptotically $\varphi_{n-1}(t) \sim p_{n-1}(t) \mathrm{e}^{-(\lambda+\omega)t}$ with $p_{n-1}(t) \in \mathbb{P}_{n-1}[t]$, the set of polynomials in $t$ of degree $n-1$. The large-time asymptotic leading order term for $\varphi_n$ can be obtained by linearising \eqref{eq:BC} around the $(0,0)$ fixed point in the $(\dot{\varphi}_n, \varphi_n)$-phase plane, which gives the linear initial value problem
    \begin{align*}
        \dot{\varphi_n}(t) + (\lambda+\omega) \varphi_n(t) = \lambda \varphi_{n-1}(t), \quad \varphi_n(0)=1.
    \end{align*}
     Solving and enforcing the initial condition, we obtain
    \begin{align*}
        \varphi_n(t) &= \left(1+ \lambda \int_0^t \varphi_{n-1}(s) \mathrm{e}^{(\lambda+\omega)s} \: ds \right) \hspace{0.1cm} \mathrm{e}^{-(\lambda+\omega)t},
    \end{align*}
    which, when using the inductive hypothesis, shows $\varphi_n(t) \sim p_n(t) \mathrm{e}^{-(\lambda+\omega)t}$. Substituting into the linear initial value problem gives an initial value problem for the polynomial factor
    \begin{align*}
        \dot{p}_n(t) = \lambda p_{n-1}(t), \quad p_n(0)=1, \quad \text{with solution} \quad p_n(t) = 1+\lambda \int_0^t p_{n-1}(s) \: ds.
    \end{align*}
    Then using $p_0(t) = \frac{\lambda+\omega}{\lambda}$ the result follows by recursion.

\subsection{Proof of \Cref{lm: uniform estimates}}\label{proof: uniform estimates}

We note that, from property (iv) of \Cref{lm: u_n properties}, $u_n \in \Cf^1(\R)$ if $n \geq 1$. Hence we restrict ourselves to $n \geq 1$ when considering continuity of $\dot{u}_n$ on $\R$.

    \subsubsection{Proof of \eqref{ineq: uniform bounds}} The first inequality follows from \eqref{eq: u_n definition} and property (i) of \Cref{lm: u_n properties}. For the second inequality we use \eqref{eq: u_n definition}, \eqref{eq: u_n odes}, and properties (i),(ii) of \Cref{lm: u_n properties}, then for $n \geq 1$ we compute
    \begin{align*}
         |\dot{u}_n(t)| = 0 \text{ if } t \leq -t_n, \quad \text{and} \quad |\dot{u}_n(t)| < \lambda+\frac{\omega}{4} \text{ if } t>-t_n,
    \end{align*}
    and the second inequality of \eqref{ineq: uniform bounds} then follows.
    \smallskip
    
    \subsubsection{Proof of \eqref{eq: u_n tail estimate t=+infty} and \eqref{eq: u_n tail estimate t=-infty}}
    From \eqref{eq: u_n odes} and property (ii) of \Cref{lm: u_n properties}, we compute
    \begin{align}\label{eq: dotu_n <= logistic}
         \dot{u}_n(t)
        &\leq - \omega u_n(t)(1-u_n(t)), \quad t \geq -t_n, \: n\geq 0.
    \end{align}
    Consider the initial value problem 
    \begin{align*}
         \dot{f}(t) = -\omega f(t) (1 - f(t)), \quad f(0)=p, 
         \end{align*}
       and its solution  
         \begin{align*}
         \quad f(t) = \frac{p}{p+(1-p)\e^{\omega t}}, \: t \in \R.
    \end{align*}
    It follows that
    \begin{align}\label{ineq: u_n comparison ode logistic}
        u_n(0)=f(0)=p, \quad \dot{u}_n(t) \leq \dot{f}(t), \quad t \geq -t_n, \: n \geq 0.
    \end{align}
    By the comparison principle for ODEs it follows from \eqref{ineq: u_n comparison ode logistic} that
    \begin{align}
        &0 \leq u_n(t) \leq f(t), \quad t \geq 0, \: n \geq 0, \label{eq: 0<=u_n<=f t>=0}\\
        &1 \geq u_n(t) \geq f(t), \quad t \leq 0, \: n \geq 0, \label{eq: 1>=u_n>=f t<=0}
    \end{align}
    where we used \eqref{eq: u_n definition} to extend the domain of the inequality \eqref{eq: 1>=u_n>=f t<=0} from $-t_n \leq t \leq 0$ to $t \leq 0$. 
    Taking the supremum over $n \geq 0$ of \eqref{eq: 0<=u_n<=f t>=0} and using $f(t) = \order(\e^{-\omega t})$ as $t \to +\infty$ gives the result \eqref{eq: u_n tail estimate t=+infty}. The inequality \eqref{eq: 1>=u_n>=f t<=0} can be rewritten as
    \[
    0 \leq 1-u_n(t) \leq 1 - f(t), \quad t \leq 0, \: n \geq 0,
    \]
    then taking the supremum over $n \geq 0$ and using $1 - f(t) = \order(\e^{\omega t})$ as $t \to -\infty$ gives the result \eqref{eq: u_n tail estimate t=-infty}.
    \smallskip

    \subsubsection{Proof of \eqref{eq: u_n^(j) tail estimate t=+-infty}}
    From \eqref{eq: u_n odes} and properties (i), (ii) of \Cref{lm: u_n properties}, we compute
    \begin{align}
        |\dot{u}_n(t)| &\leq (\lambda+\omega)|u_n(t)|, \quad t \in \R, \: n \geq 1, \label{ineq: sup|dtu_n| t+infty}\\
        |\dot{u}_n(t)| &\leq \lambda |1-u_n(t)| + \lambda |1 - u_{n-1}(t+\tau_n)| + \omega |1-u_n(t)|, \quad t \in \R, \: n \geq 1, \label{ineq: sup|dtu_n| t-infty}
    \end{align}
    where we used \eqref{eq: u_n definition} to extend the domain of definition of \eqref{ineq: sup|dtu_n| t+infty} and \eqref{ineq: sup|dtu_n| t-infty} from $t \geq -t_n$ to $t \in \R$.
    Combining \eqref{ineq: sup|dtu_n| t+infty} with the uniform estimate \eqref{eq: u_n tail estimate t=+infty} gives the result for $t \to \infty$. Using the bound in \Cref{lm: hittingtimes rel compact} and property (iii) of \Cref{lm: u_n properties}, we compute
    \begin{align*}
        1 - u_{n-1}(t+\tau_n) \leq 1 - u_{n-1}(t+T_p), \quad t \in \R, \: n \geq 1,
    \end{align*}
    where $T_p>0$ is constant. Using this inequality, we extend the inequality \eqref{ineq: sup|dtu_n| t-infty}, and then combining with
    the uniform estimate \eqref{eq: u_n tail estimate t=-infty} gives the result for $t \to -\infty$.
    \smallskip
    
    \subsubsection{Proof of \eqref{ineq: uniform dt upper bound}} For $r>0$, there exists $N_r>0$ such that $-t_n<-r$ for all $n \geq N_r$, which follows from \Cref{lm: hittingtimes rel compact}. From property (iii) of \Cref{lm: u_n properties}, we infer that $u_n$ is one-to-one on the closed interval $|t| \leq r$ for all $n \geq N_r$. From property (ii) of \Cref{lm: u_n properties}, we infer that $u_n(t)-u_{n-1}(t+\tau_n)>0$ on $|t| \leq r$ for all $n \geq N_r$. Hence, from \eqref{eq: u_n odes}, we compute
    \begin{align*}
        \sup_{n \geq N_r} \sup_{|t|\leq r} \dot{u}_n(t)
        &\leq \sup_{n \geq N_r} \sup_{|t|\leq r} - \omega u_n(t) (1 - u_n(t)) \\
        &=  \sup_{n \geq N_r} \sup_{u_n(r)\leq f\leq u_n(-r)} - \omega f (1 - f) \\
        &= \sup_{u_{N_r}(r)\leq f \leq u_{N_r}(-r)} - \omega f (1 - f)\\
        &= \max_{f \in \{ u_{N_r}(r), \: u_{N_r}(-r)\}} -\omega f(1-f) := -C_r,
    \end{align*}
    where we used that fact that $u_n$ is one-to-one on $|t| \leq r$ for all $n \geq N_r$ to justify the second line; \Cref{lm: u_n pointwise monotone} for the third line; and the maximum principle for a convex function for the last line. From \Cref{lm: u_n pointwise monotone} and the definition of $N_r$, it follows that $0<u_{N_r}(r) < p < u_{N_r}(-r)<1$,  and therefore $C_r>0$.

\end{appendix}

\begin{acks}[Acknowledgments]
We warmly thank Bernard Derrida for inspiring discussions. 
\end{acks}

\begin{funding}
The authors are grateful for generous supported by the Australian Research Council. CG is supported by the Fay Marles scholarship.
\end{funding}

\bibliographystyle{imsart-number}
\bibliography{Library}

@article{peter_blockchains_2018,
	title = {Block arrivals in the Bitcoin blockchain},
	author = {R. Bowden and H.P. Keeler and A.E. Krzesinski and P.G. Taylor},
    journal = {Stochastic Models},
    volume = {36},
	issue = {},
    pages = {602--637},
    numpages = {},
    year = {2020},
    month = {},
    publisher = {},
    doi = {},
    url = {},
}

@article{peter_blockchains_2015,
	title = {Bitcoin Blockchain Dynamics: the Selfish-Mine Strategy in the Presence of Propagation Delay},
	author = {Johannes Göbel and Paul Keeler and Anthony E. Krzesinski and Peter G. Taylor},
    journal = {Performance Evaluation},
    volume = {104},
	issue = {},
    pages = {23--41},
    numpages = {},
    year = {2016},
    month = {},
    publisher = {},
    doi = {},
    url = {},
}

@article{BarbourLuczak_2008,
	title = {Laws of Large Numbers for Epidemic Models with Countably Many Types},
	author = {A. D. Barbour and M. J. Luczak},
    journal = {The Annals of Applied Probability},
    volume = {18},
	issue = {6},
    pages = {2208--2238},
    numpages = {},
    year = {2008},
    month = {Dec},
    publisher = {Institute of Mathematical Statistics},
    doi = {},
    url = {http://www.jstor.org/stable/25442711}
}

@article{Kurtz_1980,
	title = { Relationships between stochastic and deterministic population models},
    author = {Kurtz, T},
    journal = {Biological Growth and Spread. Lecture Notes in Biomathematics},
    volume = {38},
	issue = {},
    pages = {449–467},
    numpages = {19},
    year = {1980},
    month = {},
    publisher = {Springer, Berlin, Heidelberg},
    doi = {https://doi.org/10.1007/978-3-642-61850-5_39},
    url = {https://link.springer.com/chapter/10.1007/978-3-642-61850-5_39#citeas}
}

@article{Stolyar_1996,
	title = {Asynchronous updates in large parallel systems},
	author = {Greenberg, Albert G. and Shenker, S. and Stolyar, Alexander L.},
	journal = {SIGMETRICS Perform. Eval. Rev.},
    volume = {24},
	issue = {1},
    pages = {91–103},
    numpages = {13},
    year = {1996},
    month = {May},
    publisher = {Association for Computing Machinery},
    doi = {10.1145/233013.233028},
    url = {https://doi.org/10.1145/233013.233028}
}

@article{Stolyar_2023_a,
    title={Large-scale behavior of a particle system with mean-field interaction: traveling wave solutions},
    author={Stolyar, Alexander L},
    journal={Advances in Applied Probability},
    volume={55},
    number={1},
    pages={151--186},
    year={2023},
    publisher={Cambridge University Press},
    doi={10.1017/apr.2022.24},
    url={https://doi.org}
}

@article{Stolyar_2023_b,
    title={A Particle System with Mean-Field Interaction: Large-Scale Limit of Stationary Distributions},
    author={Stolyar, Alexander L.},
    journal={Stochastic Systems},
    volume={13},
    number={3},
    pages={343--359},
    year={2023},
    publisher={INFORMS},
    doi={10.1287/stsy.2023.0108},
    url={https://pubsonline.informs.org/doi/abs/10.1287/stsy.2023.0108}
}

@article{Stolyar_2025,
	title = {A large-scale particle system with independent jumps and distributed synchronization},
	author = {Baryshnikov, Y and Stolyar, A},
	journal = {Advances in Applied Probability},
    volume = {57},
	issue = {2},
    pages = {677-707},
    numpages = {31},
    year = {2025},
    month = {June},
    publisher = {Cambridge University Press},
    doi = {10.1017/apr.2024.53},
    url = {https://www.cambridge.org/core/journals/advances-in-applied-probability/article/3AB39594D50EDC2B1B4DB685A9076877}
}

@article{LambertWfunction_1996,
	title = {On the {Lambert W} function},
    author = {Corless, Robert M and Gonnet, Gaston H and Hare, David EG and Jeffrey, David J and Knuth, Donald E},
	journal = {Advances in Computational Mathematics},
    volume = {5},
	issue = {1},
    pages = {329--359},
    numpages = {},
    year = {1996},
    month = {},
    publisher = {Springer},
    doi = {https://doi.org/10.1007/BF02124750},
    url = {https://link.springer.com/article/10.1007/BF02124750#citeas}
}

@article{MathewsEsrick_2022,
	title = {A physicist’s guide to the solution of Kummer’s equation and confluent hypergeometric functions},
	author = {W. N. MathewsJr and M. A. Esrick and Z. Y. Teoh and J. K. Freericks},
	journal = {Condensed Matter Physics},
    volume = {25},
	issue = {3},
    pages = {1-23},
    numpages = {23},
    year = {2022},
    month = {Oct},
    publisher = {Creative Commons Attribution 4.0 International License.},
    doi = {10.5488/CMP.25.33203 },
    url = {https://icmp.lviv.ua/journal/zbirnyk.111/33203/abstract.html}
}

@article{AbramowitzStegun_1970,
	title = {Handbook of Mathematical Functions with Formulas, Graphs, and Mathematical Tables},
    author = {Abramowitz, M and Stegun, I},
    journal = {},
    volume = {},
	issue = {},
    pages = {},
    numpages = {},
    year = {1970},
    month = {},
    publisher = {United States Department of Commerce, National Bureau of Standards (NBS)},
    doi = {},
    url = {},
    isbn = {0-486-61272-4}
}

@article{Rudin_1976,
	title = {Principles of Mathematical Analysis},
    author = {Rudin, Walter},
    journal = {},
    volume = {},
	issue = {},
    pages = {},
    numpages = {},
    year = {1976},
    month = {},
    publisher = {McGraw-Hill},
    doi = {},
    url = {},
    isbn = {}
}

@article{Bramson_1983,
	title = {Convergence of solutions of the Kolmogorov equation to travelling waves},
	author = {Bramson, M},
    journal = {Memoirs of the American Mathematical Society},
    volume = {44},
	issue = {},
    pages = {0-0},
    numpages = {},
    year = {1983},
    month = {},
    publisher = {},
    doi = {10.1090/MEMO/0285},
    url = {https://api.semanticscholar.org/CorpusID:123038077},
}

@article{Bramson_1978,
	title = {Maximal displacement of branching brownian motion},
	author = {Bramson, M},
    journal = {Communications on Pure and Applied Mathematics},
    volume = {31},
	issue = {},
    pages = {531-581},
    numpages = {50},
    year = {1978},
    month = {},
    publisher = {},
    doi = {10.1002/CPA.3160310502},
    url = {https://api.semanticscholar.org/CorpusID:123176293},
}

@article{DumortierPopovicKaper_2007,
      title={The critical wave speed for the Fisher–Kolmogorov–Petrowskii–Piscounov equation with cut-off},
      author={Freddy Dumortier, Nikola Popović, Tasso J. Kaper},
      journal={Nonlinearity},
      volume={20},
      number={4},
      pages={855–877},
      year={2007},
      publisher={},
      doi={10.1088/0951-7715/20/4/004}
}

@article{MuellerMytnikQuastel_2011,
      title={Effect of noise on front propagation in reaction-diffusion equations of KPP type},
      author={Mueller, C and Mytnik, L and Quastel, J},
      journal={Inventiones Mathematicae},
      volume={184},
      number={},
      pages={405–453},
      year={2011},
      publisher={Springer},
      doi={https://doi.org/10.1007/s00222-010-0292-5}
}

@article{Berard_2008,
      title={An example of {B}runet--{D}errida behavior for a branching-selection particle system on $\mathbb{Z}$},
      author={B{\'e}rard, Jean},
      journal={},
      volume={},
      number={},
      pages={},
      year={2008},
      publisher={},
      doi={https://doi.org/10.48550/arXiv.0810.5567}
}

@article{BerardGourere_2010,
      title={Brunet-Derrida Behavior of Branching-Selection Particle Systems on the Line},
      author={B{\'e}rard, Jean and Baptiste-Gour{\'e}r{\'e}, Jean},
      journal={Commun. Math. Phys.},
      volume={298},
      number={},
      pages={323-342},
      year={2010},
      publisher={Springer},
      doi={https://doi.org/10.1007/s00220-010-1067-y}
}

@article{BenguriaDepassier_2007,
      title={On the speed of pulled fronts with a cutoﬀ.},
      author={Benguria, R. D. and Depassier, M. C.},
      journal={Phys. Rev. E},
      volume={75},
      number={5},
      pages={},
      year={2007},
      publisher={American Physical Society},
      doi={https://doi.org/10.1103/PhysRevE.75.051106}
}

@article{BenguriaDepassierLoss_2008,
      title={Upper and lower bounds for the speed of pulled fronts with a cut-off},
      author={Benguria, Rafael and Depassier, M.Cristina and Loss, M.},
      journal={European Physical Journal B},
      volume={61},
      number={},
      pages={331–334},
      year={2008},
      publisher={},
      doi={https://doi.org/10.1140/epjb/e2008-00069-1}
}

@article{discreteFKPP_BramsonCorrection_2023,
      title={The logarithmic Bramson correction for
Fisher-KPP equations on the lattice $\mathbb{Z}$},
      author={C. Besse and G. Faye and J.-M. Roquejoffre
and M. Zhang},
      journal={Transactions of the
American Mathematical Society},
      volume={376},
      number={12},
      pages={8553–8619},
      year={2023},
      publisher={American Mathematical Society},
      doi={https://doi.org/10.1090/tran/9007}
}

@article{BRW_BramsonCorrection_2009,
      title={Minima in branching random walks},
      author={L. Addario-Berry and B. Reed},
      journal={The Annals of Probability},
      volume={37},
      number={3},
      pages={1044-1079},
      year={2009},
      publisher={Institute of Mathematical Statistics},
      doi={https://doi.org/10.1214/08-AOP428}
}

@article{BRW_BramsonCorrection_2013,
      title={Convergence in Law of the Minimum of a Branching Random Walk},
      author={E. A{\"i}d{\'e}kon},
      journal={The Annals of Probability},
      volume={41},
      number={3A},
      pages={1362–1426},
      year={2013},
      publisher={Institute of Mathematical Statistics},
      doi={https://doi.org/10.1214/12-AOP750}
}

@article{HamelNolen_2013,
      title={A short proof of the logarithmic Bramson correction in Fisher-KPP equations},
      author={Hamel, Fran{\c{c}}ois and Nolen, James and Roquejoffre, Jean-Michel and Ryzhik, Lenya},
      journal={Networks \& Heterogeneous Media},
      volume={8},
      number={1},
      pages={275--289},
      year={2013},
      publisher={American Institute of Mathematical Sciences},
      doi={10.3934/nhm.2013.8.275}
}

@article{BrunetDerrida_1997,
  title = {Shift in the velocity of a front due to a cutoff},
  author = {Brunet, Eric and Derrida, Bernard},
  journal = {Phys. Rev. E},
  volume = {56},
  issue = {3},
  pages = {2597--2604},
  numpages = {0},
  year = {1997},
  month = {Sep},
  publisher = {American Physical Society},
  doi = {10.1103/PhysRevE.56.2597},
  url = {https://link.aps.org/doi/10.1103/PhysRevE.56.2597}
}

@article{BrunetDerrida_1999,
  title = {Microscopic models of traveling wave equations},
  author = {Brunet, Eric and Derrida, Bernard},
  journal = {Computer Physics Communications},
  volume = {121--122},
  issue = {},
  pages = {376--381},
  numpages = {},
  year = {1999},
  month = {},
  publisher = {Elsevier},
  doi = {https://doi.org/10.1016/S0010-4655(99)00358-6},
  url = {https://doi.org/10.1016/S0010-4655(99)00358-6}
}

@article{BrunetDerrida_2001,
  title = {Effect of Microscopic Noise on Front Propagation},
  author = {Brunet, Eric and Derrida, Bernard},
  journal = {Journal of Statistical Physics},
  volume = {103},
  issue = {},
  pages = {269–-282},
  numpages = {},
  year = {2001},
  month = {},
  publisher = {Springer},
  doi = {https://doi.org/10.1023/A:1004875804376},
  url = {https://doi.org/10.1023/A:1004875804376}
}

@article{AblowitzZeppetella_1979,
	title = {Explicit solutions of Fisher's equation for a special wave speed},
	author = {Ablowitz, M and Zeppetella, A},
    journal = {Bulletin of Mathematical Biology},
    volume = {41},
	issue = {},
    pages = {835-840},
    numpages = {6},
    year = {1979},
    month = {Nov},
    publisher = {Springer},
    doi = {10.1007/BF02462380},
    url = {https://link.springer.com/article/10.1007/BF02462380},
}

@article{AronsonWeinberger_1978,
	title = {Multidimensional nonlinear diffusion arising in population genetics},
	author = {Ablowitz, M and Zeppetella, A},
    journal = {Advances in Mathematics},
    volume = {30},
	issue = {1},
    pages = {33-76},
    numpages = {6},
    year = {1978},
    month = {},
    publisher = {},
    doi = {10.1016/0001-8708(78)90130-5},
    url = {https://www.sciencedirect.com/science/article/pii/0001870878901305},
}

@article{DeeLanger_1983,
  title = {Propagating Pattern Selection},
  author = {Dee, G. and Langer, J. S.},
  journal = {Phys. Rev. Lett.},
  volume = {50},
  issue = {6},
  pages = {383--386},
  numpages = {0},
  year = {1983},
  month = {Feb},
  publisher = {American Physical Society},
  doi = {10.1103/PhysRevLett.50.383},
  url = {https://link.aps.org/doi/10.1103/PhysRevLett.50.383}
}

@article{ColletEckmann_1990,
  title = {Instabilities and Fronts in Extended Systems},
  author = {P. Collet and J.-P. Eckmann},
  journal = {Princeton Series in Physics},
  volume = {},
  issue = {},
  pages = {},
  numpages = {0},
  year = {1990},
  month = {Jun},
  publisher = {Princeton University Press, Princeton},
  doi = {},
  url = {https://press.princeton.edu/books/hardcover/9780691636177/instabilities-and-fronts-in-extended-systems?srsltid=AfmBOop9KjZGahQjovdCBmnx-MhGnS-iGz1Zaaw-C5I7EKspN6lpfJJN}
}

@article{vanSaarloosWim_1987,
  title = {Dynamical Velocity Selection:
Marginal Stability},
  author = {van Saarloos, Wim},
  journal = {Phys Rev Lett.},
  volume = {58},
  issue = {12},
  pages = {2571–2574},
  numpages = {0},
  year = {1987},
  month = {Jun},
  publisher = {},
  doi = {10.1103/PhysRevLett.58.2571},
  url = {https://pubmed.ncbi.nlm.nih.gov/10034786/}
}

@article{vanSaarloosWim_1988,
  title = {Front propagation into unstable states: Marginal stability as a dynamical mechanism for velocity selection},
  author = {van Saarloos, Wim},
  journal = {Phys. Rev. A},
  volume = {37},
  issue = {1},
  pages = {211--229},
  numpages = {0},
  year = {1988},
  month = {Jan},
  publisher = {American Physical Society},
  doi = {10.1103/PhysRevA.37.211},
  url = {https://link.aps.org/doi/10.1103/PhysRevA.37.211}
}

@article{vanSaarloosWim_1989,
  title = {Front propagation into unstable states. II. Linear versus nonlinear marginal stability and rate of convergence},
  author = {van Saarloos, Wim},
  journal = {Phys. Rev. A},
  volume = {39},
  issue = {12},
  pages = {6367--6390},
  numpages = {0},
  year = {1989},
  month = {Jun},
  publisher = {American Physical Society},
  doi = {10.1103/PhysRevA.39.6367},
  url = {https://link.aps.org/doi/10.1103/PhysRevA.39.6367}
}

\end{document}